\documentclass{elsarticle}
\usepackage{geometry}
\usepackage{bm}
\usepackage{amsmath}
\usepackage{amsthm}
\usepackage{amssymb}
\usepackage{tikz}
\usepackage{mathtools}

\usepackage{float}
\usepackage{tabularx}

\usepackage{caption,subcaption}
\usepackage{color}

\usepackage{tikz}
\newtheorem{theorem}{\textbf{ Theorem} }
\newtheorem{lemma}{\textbf{Lemma } }
\newtheorem{definition}{\textbf{Definition} }
\newtheorem{corollary}{\textbf{Corollary} }
\newtheorem{remark}{\textbf{Remark} }

\newtheorem{property}{\textbf{Property} }
\newtheorem{conjecture}{\textbf{Conjecture} }
\begin{document}
	
	\title{
		Locally  Optimal   Eigenvectors of  Regular  Simplex Tensors
		\tnoteref{mytitlenote}
	}
	\tnotetext[mytitlenote]{
		This work was partially supported by National Natural Science Fund of China (62271090), Chongqing Natural Science Fund (cstc2021jcyj-jqX0023), National Key R\&D Program of China (2021YFB3100800), CCF Hikvision Open Fund (CCF-HIKVISION OF 20210002), CAAI-Huawei MindSpore Open Fund, and Beijing Academy of Artificial Intelligence (BAAI).
	}
	
	\author[add1]{Lei  Wang\corref{coraut}}
		\cortext[coraut]{Corresponding author}
	\ead{wanglei179@mails.ucas.ac.cn}

		\address[add1]{	the School of Microelectronics and Communication
	Engineering, Chongqing University, Chongqing 400044, China}

%
%
%


\begin{abstract}
Identifying  locally     optimal 
solutions 
is an important issue given an   optimization  model.
In this paper, 
we focus on  a  special class of  symmetric tensors  termed  regular simplex tensors,  which is  
   a newly-emerging concept, 
and 
 investigate    its  local optimality  of the related  constrained  nonconvex  optimization model.
 This   is proceeded by  checking the first-order and  second-order necessary condition sequentially.  
Some  interesting directions concerning the regular simplex  tensors, 
including  the  
robust eigenpairs checking   and other  potential issues, 
are  discussed in the end   for    future work.



\end{abstract}


\maketitle


\section{Introduction }
Matrix  analysis and   applications  have been researched for more than 
150 years,
and with very  fruitful  achievements in  both  theoretical  results and  practical applications.
As  the high-order  generalization of  matrix,  tensor   analysis and  applications have  been  given   more   attentions in  recent  years \cite{kolda,2007Numerical,hosvd,tensorrank_qi,TensorPCA,RobustTensorCompletion,TensorDiagonalization,TensorApproximation}.

 Naturally, 
 the concept of  eigen-analysis for  second-order matrix  
has also    been    extended  into    higher-order  tensors.
Different  from  the matrix  case,    there are several  definitions for eigenpairs (a pair refers to the eigenvalue and it   corresponding eigenvector)  of symmetric tensors, such as   D-eigenpairs \cite{Deigen}, H-eigenpairs  \cite{qi},  E-eigenpairs  \cite{lim,qi},   etc.  
In this paper,  we mainly  focus  on  one of them ---  E-eigenpairs, whose eigenvector is  subject to    unit length.  
Without causing  ambiguities, 
in the whole context, 
we use  the  term eigenpair  to refer to the focused  E-eigenpairs  for simplicity. 
Such a concept has been widely exploited  in  theory \cite{SHOPM,ASHOPM,NCM,OTD,Cuicf,hm,Zglobal}  and also  
made 
numerous  applications   in  many disciplines, such as latent  variable mode \cite{Hsu}, hyperspectral  image processing \cite{PSA, NPSA,MSDP}, 
signal processing \cite{tensor_bss1},  hypergraph theory \cite{hypergraph,hypergraph3} and so on.


However,  it  has  also  been   shown  that  most of tensor  problems  are NP hard  \cite{NP-hard},  including computing all  eigenpairs of tensors.
To our best knowledge,  
so far,  there are  only  two  algorithms that  can  obtain  all  eigenpairs of tensor \cite{Cuicf,hm}.
However, 
both of them  still  suffer from high   computational  complexity for large-scale tensor. 
Therefore, most of  previous works aim
 to obtain the maximized or minimized one eigenpairs, and different optimization algorithms were developed, such as \cite{SHOPM,ASHOPM,NCM}.

In  this sense, previous works have 
turned to  paying  attentions to some   special  classes  of  symmetric  tensors,  
and analyzed the  eigen-problem. 
One   widely researched   type   with fruitful results 
is termed 
  orthogonally  decomposable  (odeco) tensors \cite{odst,OTD,Hsu,sr1,cunmu,GloballyConvergent},  
  which 
  is  
  an natural generalization of orthogonal matrix decomposition.
It has been shown that
concerning the odeco tensors, 
the number of 
its 
all  eigenpairs can reach at the upper bound\cite{odst},
and the locally maximized eigenpairs correspond to the 
orthogonal basis.


However, 
unfortunately, 
most of symmetric tensors cannot be 
orthogonally  decomposable. 
In  this sense, 
recently, 
some researchers further  extended 
the case of  odeco tensors into 
a  more generalized one, where the  symmetric  tensor  is  generated  by 
the set of some  equiangular  set (ES) or equiangular tight frame (ETF), which contains 
 of $r$ vectors in 
$n-1$-dimensional space \cite{RobustEigen}. 
For  example, 
the case of  $r=n-1$  serves as a  special one of tight frame,  which  
forms a  standard orthonormal  basis and  corresponds to 
the  odeco tensors.
When $r=n$,  the 
frame is termed the regular simplex one, 
and the  generated tensor is thus called regular simplex tensor. 
In  \cite{RobustEigen}, 
they  discussed that under what  condition 
the eigenvectors will be a robust one.
Later, 
several researchers 
investigated 
the real eigen-structure 
of  all eigenpairs
by  analyzing the first-order 
necessary condition \cite{teneigenstructure}.

However, 
the  local 
optimality of regular simplex tensor,
which is determined by the second-order  necessary condition, 
has not been answered. 
In this paper, 
different from 
the previous works that focused on the original  optimization model, 
we further 
focused on this issue 
and 
equivalently  developed 
a  new reformulated model  to analyze 
its 
local optimality
by checking the first-order and second-order necessary condition 
sequentially.

The  rest of the paper
is organized as  follows.
In Section  \ref{pre},
Some preliminaries related to the subject are provided,  including 
the  definition  for  
tensor  eigenpairs,  and  the  focused regular simplex tensor. 
In  Section  \ref{reformulated},
the 
optimization model  for   regular simplex  tensor eigenpairs
 is  
 reformulated as a equivalent   one with better separable  structure, which is beneficial 
 for  analyzing the eigen-structure of tensor eigenpairs.
 Then, in Section  \ref{station} and \ref{locallymaximized},
 the structure for all  stationary 
 and  
 the locally optimal  solutions  of 
 the newly reformulated model 
 is investigated 
 by checking the first-order and second-order necessary condition   sequentially.
 Some future works are discussed in Section  \ref{futurework}.

\section{Background}\label{pre}
In this part, 
some 
background  preliminaries   related to the subject are provided,  including 
the used notations, 
the  definition  for  
tensor  eigenpairs,  and  the  focused regular simplex tensor. 
The  details are as  follows.

\subsection{Preliminaries}\label{notation}
We    introduce  some  necessary   notations, definitions   and  lemmas used  in   this  article. 
In  this  material,  as adopted in  many tensor-related works \cite{kolda,TensorPCA,9521829},  high-order tensors are denoted
in
boldface Euler script letters, e.g.,
$\mathcal  A $.
Matrices are denoted  in   boldface capital letters, e.g., $\mathbf  A $; vectors are denoted in  
boldface lowercase letters, e.g.,  $\mathbf  a $.
Sets and subsets are denoted  in blackboard bold  capital letters, e.g.,  $\mathbb  A $.

A  $d$th-order tensor is denoted 
$\mathcal A \in \mathbb {R}^{I_1 \times I_2  \times \dots \times I_{d} }$,
where 
$d$
is the order   of
$\mathcal A $, and 
$ I_j $ ($  j \in \{ 1,2,\dots,d \}$)  is  the  dimension  of  
$j$th-mode.
The element of $\mathcal A$,    which  is  indexed  by integer tuples $(i_1,i_2,\dots,i_d) $, is denoted 
($a_{i_1,i_2,\dots,i_d})_
{1 \le i_1 \le I_1, 
	\dots, 
	1 \le i_d \le I_d}
$. 
A   tensor is called    symmetric if its elements remain invariant under any permutation of 
the  indices\cite{kolda}. 
Let    $  T^{m}(\mathbb R^{n}) $ denotes   the  space  of  all  such  real  symmetric    tensors.
Given a $m$th-order  $n$-dimensional symmetric  tensor  $\mathcal S $ and  a  vector $ \mathbf u \in \mathbb {R}^{n \times 1}$, we have
$ 
\mathcal S \mathbf u^{m} =
\sum\limits_{i_1,i_2,\dots,i_m=1}^{n} 
s_{i_1,i_2,\dots,i_m}  u_{i_1} \dots   u_{i_{m}}
$, 
and  $   \mathcal S \mathbf u^{m-1}   $   denotes   a    $n$-dimensional
column   
vector,  whose  $j$th  element   is    
$
(\mathcal S \mathbf u^{m-1})_{j} =
\sum\limits_{i_2,\dots,i_m =1}^{n} 
s_{j,i_2,\dots,i_m}  u_{i_2} \dots   u_{i_{m}}
$\cite{Cuicf}.
Furthermore,   $   \mathcal S \mathbf u^{m-2}   $   is  an    $n  \times  n $  matrix,    whose  $(i, j )$th  element   is  
$
(\mathcal S \mathbf u^{m-2})_{i, j} =
\sum\limits_{i_3,\dots,i_m =1}^{n} 
s_{i,j,i_3,\dots,i_m}  u_{i_3} \dots   u_{i_{m}}.
$

$\mathbf  1_{n}$   denotes   a  $ n  \times  1$ column vector,  and  $\mathbf  I_{n}$   denotes   a  $ n  \times  n$  identity  matrix.
$\mathbf  J_{m \times p}$    and  $\mathbf  J_{m}$    is   a   $ m  \times  p$  matrix  and   a  $ m  \times  m$    square    matrix,     with all  elements equal to 1,  respectively. 
It holds   that\cite{zhang2017matrix}
\begin{equation}\label{Jproperty}
\mathbf  J_{m \times p} = \mathbf  1_{m}  \mathbf  1_{p}^{\mathrm {T}}. 
\end{equation}
Similarly,  $\mathbf  0_{n}$,  $\mathbf  O_{m \times p}$    and  $\mathbf  O_{m}$   are   matrices   with all  elements equal to 0.

$\triangledown$ is  the  gradient  operator.
$ \rm Null( \mathbf A) $ 
denotes the null space  of $\mathbf A$. 
$ \mathbf  P_{\mathbf  A }^{\bot} $ is  the  orthogonal   cpmplement  operator  of  $\mathbf A$.
$``diag(\mathbf u)" $ is an operator which maps the vector  $  \mathbf u  \in  \mathbb R^{n \times 1} $   to 
a $ n \times n $ 
diagonal matrix with its  diagonal  elements to be that of $  \mathbf u$.
$``$ $\circledast$ $"$ is an  elementwise   multiplication  operation, where $   (\mathbf  {A} \circledast  \mathbf  {B})_{ij} =
\mathbf  A_{ij}\mathbf  B_{ij}$.
  For simplicity,  we use
$ \mathbf A ^{\circledast^{p}}$
and
$ \mathbf a ^{\circledast^{p}}$
to denote the $p$-times  elementwise   multiplication  of the matrix
$ \mathbf A  $ and the vector $ \mathbf a $,  respectively.

\begin{definition}[\textbf{outer product}]
	\label{outerprod}
	Given   $m$  vectors 
	$ \mathbf a^{(i) } \in \mathbb {R}^{I_i \times 1}$ 
	($i=1,2, \dots, m$),
	their   outer  product   
	$ \mathbf a^{(1) }
	\circ
	\mathbf a^{(2) }
	\circ  \dots
	\circ
	\mathbf a^{(m) } 
	$  
	is   a    $ m$th-order  tensor denoted $   \mathcal A$,  with  a size  of  
	$ I_1 \times I_2  \times \dots \times I_m  $.  
	And  its    element is    the  product  of    the  corresponding  vectors'   elements, i.e., 
	$ 
	a_{i_1,i_2,\dots,i_m}
	= 
	\mathbf a^{(1) }_{i_1}
	\mathbf a^{(2) }_{i_2}
	\dots
	\mathbf a^{(d) }_{i_m} 
	.
	$
	When 
	$  \mathbf a^{(1) }
	=  
	\mathbf a^{(2) }
	=   \dots
	= 
	\mathbf a^{(m) }
	=\mathbf a $ $(I_1 =  I_2  = \dots = I_m =I)$,  we use  the  notation 
	$ \mathcal A =  \mathbf a^{\circ m}$  for  simplicity,  where 
	$   \mathcal A $ is  a  symmetric  tensor  of  order  $m$  and  dimension  $ I$.  
\end{definition}

\begin{definition}[\textbf{direct  sum}]
	The  direct  sum  of   a   $ n \times n$  matrix  $\mathbf  {A}  $ and  a   $m \times m$ matrix   $\mathbf  {B} $  is a  matrix  with a  size  of  $(n+m) \times(n+m)$,  denoted    $\mathbf  A \oplus \mathbf  B$, which  follows: 
	
	\begin{equation}
	\mathbf  {A} \oplus \mathbf  {B}=
	\left[\begin{array}{cc}
	\mathbf {A} & \mathbf {O}_{n \times m} \\
	\mathbf {O}_{m \times n} & \mathbf {B}
	\end{array}\right] .
	\end{equation}
\end{definition}

\subsection{Optimization  theories of  Tensor eigenpairs}
In this part,  we briefly  introduce the  optimization theories   related to the tensor  eigenpairs problem. 
The  concept  of  tensor  eigenpairs  can be    understood  and  derived  by  considering  the  following   constrained  nonconvex optimization  model:
\begin{equation}\label{opti_ori}
\begin{cases}
\max\limits_{\mathbf v} \quad \mathcal S \mathbf v^{m}   \\
\rm s.t. \quad \mathbf v^{\mathrm {T}}\mathbf v=1
\end{cases}.
\end{equation}
The Lagrangian function 
of (\ref{opti_ori})  is  defined as:
\begin{equation}\label{Lagrangian_function}
L(\mathbf v, \lambda)=
\frac {1}  {\it m}
\mathcal S \mathbf v^{\it m}
+
\frac { \lambda} {2} (1-  \mathbf v^{\mathrm {T}}\mathbf v).
\end{equation}
When the gradient of
$  L(\mathbf v, \lambda) $ to
$ \mathbf  v $ is  $ \mathbf 0$,    the eigenpair of a  symmetric  tensor can  be  deduced,  which  was  independently    defined  by   Lim  and  Qi  in  2005:

\begin{definition} \cite{qi,lim} [\textbf{eigenpairs of  symmetric tensor}]
	Given a  tensor $\mathcal S   \in    T^{m}(\mathbb R^{n}) $,
	a pair
	$(\lambda ,\mathbf v )$
	is an  eigenpair  of  
	$\mathcal S  $ 
	if
	\begin{equation}\label{definition}
	\mathcal S \mathbf v^{m-1}=\lambda \mathbf v,
	\end{equation}
	where
$ \lambda  \in  \mathbb C $
is  the  eigenvalue and
$ \mathbf v  \in   \mathbb C^{n \times  1} $
is the  corresponding   eigenvector   satisifying 
$\mathbf v^{\mathrm {T}}\mathbf v=1 $.
\end{definition}

The  first-order    gradient  derivation     (\ref{definition})   can be  used  to    obtain        all   stationary  points  of  (\ref{opti_ori}), i.e., 
all eigenpairs of tensor.   
Note  that   not all of the  stationary  points are  the  locally  optimal  points. 
Some of these    can be  categorized as the saddle points.  
While   the  second-order  derivation  information  plays  an  important  role  in  identifying    whether  a  stationary  point  is     locally    optimal   
or saddle  given    an  optimization   model.      The    second-order derivation  
of
$  L(\mathbf v, \lambda) $ 
to    $ \mathbf v $, which is  also  termed  the Hessian matrix of
(\ref{Lagrangian_function}),  is   denoted
by 
\begin{equation}\label{hessian_matrix}
\mathbf H(\mathbf v) = (m-1)\mathcal S \mathbf v^{m-2} - \lambda \mathbf I_{n-1} ,
\end{equation}
where
$\mathcal S \mathbf v^{m-2} $ was defined in Subsection \ref{notation},
and 
$ \mathbf I_{n-1} $
is an $ (n-1) \times (n-1) $  identity matrix.

Then,  
the 
locally    optimal   solutions of   (\ref{opti_ori})
can  be  identified  by
checking the  negative 
definiteness  
and determining the sign of each  eigenvalue
of the  following  matrix:
	\begin{align}\label{Mhess}
\mathbf {K}
& =  \mathbf  P_{\mathbf  v}^{\bot} \mathbf H (\mathbf v)   \mathbf  P_{\mathbf  v}^{\bot}
	\end{align}
	and the  detailed  explanation for  the  
	derivation of 
	(\ref{Mhess})
	can refer to the Appendix part.

		\subsection{Regular  simplex  frame  and  tensor}
		In this part,  we  will  introduce  a  special class of symmetric   tensors,  which is termed   
		regular  simplex  one. 
		First,  the    definition of the  generalized  equiangular set  is  introduced as  follows:  
		\begin{definition}[Definition 4.1  in Ref. \cite{RobustEigen}]
		An equiangular set (ES) is a collection of vectors $\mathbf{w}_{1}, \ldots, \mathbf{w}_{r} \in \mathbb{R}^{n-1}$ with $r \geq n-1$ if there exists $\alpha \in \mathbb{R}$ such that
		$$
		\alpha=
		\left|\left\langle\mathbf{w}_{i}, \mathbf{w}_{j}\right\rangle\right|, \forall i \neq j \quad \text { and } \quad\left\|\mathbf{w}_{i}\right\|=1, \forall i .
		$$
		Furthermore, an $\mathrm{ES}$ is  called  an equiangular tight frame (ETF)  if  
		\begin{equation}
	\mathbf{W} \mathbf{W}^{\mathrm T}=a \mathbf{I}_{(n-1)}, \quad \mathbf{W}:=\left(\mathbf{w}_{1},  \cdots,  \mathbf{w}_{r}\right) \in \mathbb{R}^{(n-1) \times r}
	\end{equation}
	\end{definition}

	For  example, 
	when   $r=n-1$ and $a=1$, the  orthonormal bases
	 $\left\{\mathbf{w}_{1}, \ldots, \mathbf{w}_{n-1}\right\} \subset \mathbb{R}^{(n-1) \times (n-1)}$
	 forms  an  ETF,  where $\alpha =0$.
	 When    $r=n$ and $a=\frac {n}{n-1}$,   $\left\{\mathbf{w}_{1}, \ldots, \mathbf{w}_{n}\right\} \subset \mathbb{R}^{(n-1) \times  n}$, it  is   termed  
	 regular  simplex frames,  
	 where   $ 
	 \alpha = - \frac{1}{n-1}$.
	 $\mathbf{w}_{i}, i=1,2 \dots, n$   are  called  the  vectors  in  the  frame. 
	 In  2D  space,  
	 the  regular simplex  frame  is  a   regular  triangle.

	 In  addition, for the regular  simplex frame,  the  following property  holds,  which  will be  useful in  the  later  section: 
	 
	 \begin{property}\label{regularproperty}
	For  the  regular  simplex   frame,  the nullspace of 
	 $ \mathbf{W}^{\mathrm T} \mathbf{W} $
	  and 	 $  \mathbf{W} $
	   is spanned by 
	   $ \mathbf{1}_{n} $,
	   i.e., 
	    $ 
	    \mathbf{W}^{\mathrm T} \mathbf{W}  \mathbf{1}_{n} 
	    =
	     \mathbf{W}  \mathbf{1}_{n} 
	     =
	     \mathbf{0}_{n-1}
	    $.
	   	 \end{property}

		Then,  the  regular  simplex  tensor  is   one  deduced  by  the  the  regular  simplex   frame  with  the  following   form:
		\begin{equation}\label{simplextensor}
		\mathcal{S}:=\sum_{i=1}^{n} \mathbf{w}_{i}^{\circ m}
		\end{equation}
		where $\circ$ is the outer product  defined  in Definition \ref{outerprod}, 
and  	$\mathcal{S}
	$ is  a   symmetric  tensor of order
		$ m $
		and dimension
		$ n-1$.   
		Note that
		when the tensor is generated by the orthonormal bases
		$\left\{\mathbf{w}_{1}, \ldots, \mathbf{w}_{n-1}\right\} \subset \mathbb{R}^{(n-1) \times (n-1)}$,
		the corresponding tensor is called an orthogonally  decomposable (odeco).
		Clearly, the regular simplex tensor can be viewed as a generalization of the odeco tensors.

In  this paper, 
we will mainly focus on such a type of 	 
regular simplex tensor and tend to  analyze its eigen-stucture, including its stationary and locally optimal solutions. 
Before our discussion,	some necessary remarks   concerning  the discussed  scope    are  made as follows:
	\begin{remark}\label{RemarkScope}
		
		\begin{itemize}
			\item The notation  adopted   in this paper  is slightly  different from the previous work \cite{RobustEigen,teneigenstructure}.
			The existing ones considered  $n+1$ vertices in $n $ dimensional space, while  here we
			considered  $n$ vertices in $n-1 $ dimensional space. 
			Such a  modification is done for a convenient  formula  derivations  in the later  section (see Section  \ref{reformulated} for details), which is  trivial.
				\item 
				Throughout this paper,  we  only  concern with the case where 
				$n \ge 3, m \ge 3$  for (\ref{simplextensor}). 
				The matrix case (corresponding to 	$ m = 2$)  was not included in the discussion, since it can be checked that 
				when $m=2$, the second-order tensor for (\ref{simplextensor}) will be 
				proportional to an  identity matrix, up to a constant factor, whose  eigen-structure is  obvious.
				In addition,  it should also be noted that 
				the case  of 		$n = 3, m = 4$   turns to be  a very  special combination, 
				whose corresponding   objective  value, i.e., 
				$
			\mathcal S \mathbf v^{m}, 
			$   is a constant.
				Such a  conclusion has been analyzed in Theorem 4.6 of \cite{RobustEigen}  
				 and Theorem 12 of \cite{teneigenstructure}.
				Please refer for  details. 		
				In the later analysis, 
				this combination is also always excluded out of the discussed scope. 	
		\end{itemize}
	\end{remark}
		

\section{Problem  reformulation for   regular simplex  tensor eigenpairs}\label{reformulated}

In  this  part, we   first  transform  the  optimization  model     by  
reformulating  a  new  model  for  the  regular  simplex  tensor.
The  details are as  follows.

 Directly  focusing  on  the  original  optimization model   (\ref{opti_ori}) where  the  tensor  is  given  by  (\ref{simplextensor})
 may be  a  complicated   task.
 In  this  part,  by  utilizing  
 Property  \ref{regularproperty}, 
 an equivalent   optimization  model  is  reformulated  to   deal  with  the  issue. 
 First,  
 the  objective function   when $ \mathcal{S}$ is  given by  (\ref{simplextensor}) can be  rewritten as:
 	\begin{equation}
 \mathcal{S}  \mathbf v^{m}    
  =
 (\sum_{i=1}^{n} \mathbf{w}_{i}^{\circ m} )\mathbf v^{m} 
 =
  \sum_{i=1}^{n} ( \mathbf v^{\mathrm T} \mathbf{w}_{i})^{ m} .
 \end{equation}
 Denote  
 \begin{align}\label{udenote}
 \mathbf   u =
 \sqrt {
 	  \frac  {n-1}{ n}
  }   
 {\mathbf W}^{\mathrm T}  {\mathbf v}  = 
 [u_1,u_2,\dots, u_{n}]^{\mathrm T} \in  \mathbb {R}^{n \times 1} ,
 \end{align} 
 it holds that
 	\begin{equation}
 \mathcal{S}  \mathbf v^{m}    
\varpropto 
\sum\limits_{i=1}^{n}  u_{i}^{m}.
 \end{equation}
Furthermore, 
we  can  derive  that 
  \begin{align}\label{utu}
 \mathbf   u ^{\mathrm T} \mathbf   u
  =
   \frac  {n-1}{  n  }  
   {\mathbf v} ^{\mathrm T} {\mathbf W} {\mathbf W}^{\mathrm T}  {\mathbf v}  
   =
   {\mathbf v} ^{\mathrm T}  {\mathbf v}  
   =1 ,
 \end{align}
 and 
   \begin{align}\label{uln}
 \mathbf   u ^{\mathrm T}
  \mathbf 1_{n}
 =
 \sqrt {
	\frac  {n-1}{ n}
}   
 {\mathbf v} ^{\mathrm T} {\mathbf W}   {\mathbf 1_{n}}  
 = 
0 .
 \end{align}
 
Then,  model   (\ref{opti_ori})   for the regular simplex tensor, which is with the following form 
 \begin{align}\label{optmodelori}
\begin{cases}
\max\limits_{\mathbf v} \quad \mathcal S \mathbf v^{m} 
=
(\sum\limits_{i=1}^{n} \mathbf{w}_{i}^{\circ m}) \mathbf v^{m}   \\
\rm s.t. \quad \mathbf v^{\mathrm {T}}\mathbf v=1
\end{cases}
\end{align}
  can  be  equivalently    transformed into the following 
 model:  
 \begin{align}\label{optmodel}
 \begin{cases}
 \max\limits_{\mathbf u}  \quad  \sum\limits_{i=1}^{n}  u_{i}^{m} 
  \\
 \text { s.t. }   \quad    
 \mathbf   u^{\mathrm T}   \mathbf   u =1 ,
 \quad  
 \mathbf   u^{\mathrm T}   \mathbf   1_{n} =0
 \end{cases} .
 \end{align}
 
%
%
%

 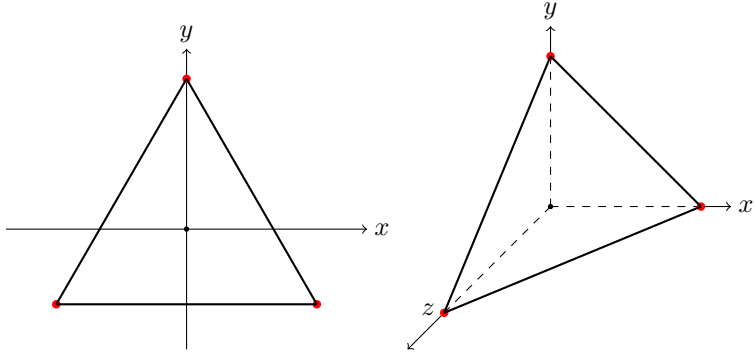
\begin{figure}
 	\centering
 	\begin{tikzpicture}[scale=2]
 \draw[->] (-1.2,0)  -- (1.2,0);
 \draw[->]  (0,-0.8)  -- (0,1.2);
 	\fill  (0:0)  circle(0.5pt);
 	
 	\node (P)  at (0:1.3) {$x$} ;
 	\node (P)  at (90: 1.3) {$y$} ;
 	
 	\fill[color=red]  (0,1)  circle(0.8pt);
 	\fill[color=red]  ( 1.7321/2,  -0.5)  circle(0.8pt);
 	\fill[color=red]  (-1.7321/2,  -0.5)    circle(0.8pt);

 	\draw[color=black,thick] (0,1) -- ( 1.7321/2,  -0.5) ; 
 	\draw[color=black,thick] (0,1)  -- ( -1.7321/2,  -0.5)  ;
 	\draw[color=black,thick] ( 1.7321/2,  -0.5)  --( -1.7321/2,  -0.5)  ;

 		
 	\end{tikzpicture}
 	 	\begin{tikzpicture}[scale=2]
 	\draw[->] (1,0)  -- (1.2,0);
 	\draw[->]  (0,1)  -- (0,1.2);
 	\draw[->]  (-0.7071   ,-0.7071  )  -- (-0.95  ,-0.95 );
 	
 	 	\draw[color=black,dashed] (0,0)  -- (1,0);
 	\draw[color=black,dashed]  (0,0)  -- (0,1);
 	\draw[color=black,dashed]  (0,0)  -- (-0.7071   ,-0.7071  );

 	\fill  (0:0)  circle(0.5pt);
 	
 	\node (P)  at (0:1.3) {$x$} ;
 	\node (P)  at (90: 1.3) {$y$} ;
 		\node (P)  at (220: 1.06) {$z$} ;
 	
 	\fill[color=red]  (0,1)  circle(0.8pt);
 	\fill[color=red]  ( 1,  0)  circle(0.8pt);
 	\fill[color=red]  (-0.7071 ,  -0.7071 )    circle(0.8pt);

 	\draw[color=black,thick] (0,1) -- ( 1,  0) ; 
 	\draw[color=black,thick] (0,1)  -- (-0.7071 ,  -0.7071 ) ;
 	\draw[color=black,thick] (-0.7071 ,  -0.7071 )    -- ( 1,  0)  ;


 	\end{tikzpicture}
 		\caption{
 		An intuitive  sketch map 
 	concerning the reformulation  from model (\ref{optmodelori})  to 
 	(\ref{optmodel})  for the case of $n=3$:
 transforming  the regular triangle  from   2D    to  3D  space.
 }
	
 	\label{simplextrans}
 \end{figure}

 For   simplicity,  in the later  analysis, denote  
\begin{align}
\notag
 & f (\mathbf u) = \sum_{i=1}^{n}  u_{i}^{m}, \\ 
&\nonumber   g_{1}(\mathbf u) =  \mathbf u^{\mathrm T}\mathbf u -1 = 0,  \\
&\nonumber  g_{2}(\mathbf u) = \mathbf u^{\mathrm T}\mathbf 1_{n}-0 = 0 .  
 \end{align}

 The  reformulated  model  can  be 
 understood as    a  transformation   that  transfers    the  original model  in $n-1$-dimensional  space  into 
 the $n$-dimensional one  for  analysis. 
And  the  new  constraint  
 $  \mathbf   u^{\mathrm T}   \mathbf   1_{n} =0 $
 is  naturally  related  to   Property  \ref{regularproperty}.
 	An intuitive  sketch map 
 for the case of $n=3$  is  plotted  in  Fig~\ref{simplextrans}.
 
 It should be emphasized that 
 there exists an one-to-one  correspondence relationship 
 between 
 the solutions of 
 (\ref{optmodelori})  and
 those of 
  (\ref{optmodel}).
 When 
 $ \mathbf v$ is a feasible solution of  (\ref{optmodelori}), 
 the corresponding $\mathbf u$ calculated by 
 (\ref{udenote}) 
 will also be a solution of (\ref{optmodel}), and vice verse. 
 Compared to the original  model (\ref{optmodelori}), 
   (\ref{optmodel}) 
   could be 
   with the  following advantage:
   the objective function $f (\mathbf u)$  is 
   separable regarding to the $n$ variables  $u_{i}$, 
  which will be beneficial  to  analyzing 
  the structure 
  for eigenpairs, and this  will be discussed in the   next  two  sections. 
 
 Then, in the  following Section  \ref{station} and \ref{locallymaximized},
 the structure for all  the  stationary  and  locally optimal  solutions of 
 the newly reformulated model 
 is investigated  to  sequentially proceed the analysis.

 \section{Structure  of  all  stationary points of  (\ref{optmodel})}\label{station}
The Lagrangian function  of the  reformulated  constrained  model in  (\ref{optmodel})    is 
defined as:
\begin{equation}\label{Lagrangianfunctionnew}
L(\mathbf u, \alpha, \beta)=
\frac{1}{m}
\sum\limits_{i=1}^{n}  u_{i}^{m}						
+
\frac{\alpha}{2}
(1- \mathbf   u^{\mathrm T}   \mathbf   u  )
+
\beta (0- \mathbf   u^{\mathrm T}   \mathbf   1_{n}),
\end{equation}
where $ \alpha,\beta$ are  the  corresponding   Lagrangian   multipliers of  the two  constraints.

To  obtain all the stationary points, one can  check  the first-order  necessary   condition
(also known as Karush-Kuhn-Tucker (KKT) condition)
of model  (\ref{optmodel}).
The gradient of
$  L(\mathbf u, \alpha, \beta) $    with  respect  to  
$ \mathbf u $   follows: 
\begin{equation}\label{gradient}
\triangledown_{\mathbf u}L =
\mathbf u ^{\circledast^{m-1}}
-
\alpha \mathbf u-\beta \mathbf  1_{n} . 
\end{equation}

 When
$ \triangledown_{\mathbf u}L = \mathbf 0$,  it  holds  that 
\begin{equation}
\notag
\mathbf u ^{\circledast^{m-1}}-\alpha \mathbf u-\beta \mathbf  1_{n} = \mathbf 0 . 
\end{equation}

To   conclude,  the  
stationary points of  (\ref{optmodel})
should  simultaneously 
satisfy 
  the  following  three  equations:  
\begin{align}
\label{KKTgradient}
&\mathbf u ^{\circledast^{m-1}}-\alpha \mathbf u-\beta \mathbf  1_{n} = \mathbf 0 . 
\\
\label{KKTcon1}
&\mathbf   u^{\mathrm T}   \mathbf   u =1 .
\\
\label{KKTcon2}
&\mathbf   u^{\mathrm T}   \mathbf   1_{n} =0.
\end{align}

By multiplying  $\mathbf   u^{\mathrm T}$ on  both sides of (\ref{KKTgradient}) and utilizing  (\ref{KKTcon1})  and
(\ref{KKTcon2}), 
we can  obtain  that 
\begin{align}\label{alphares}
\alpha 
=
\sum\limits_{i=1}^{n}  u_{i}^{m}
=
f(\mathbf u).	
\end{align}

Similarly, by multiplying  $  \mathbf   1_{n}^{\mathrm T}$ on  both sides of (\ref{KKTgradient}) and utilizing  (\ref{KKTcon1})  and
(\ref{KKTcon2}), 
 we can  obtain  that 
\begin{align}\label{beatres}
\beta 
=
\frac 1 n  \sum\limits_{i=1}^{n}  u_{i}^{m-1}.	
\end{align}

\begin{figure*}[t]
	\centering
	\begin{subfigure}[htbp!]{0.39\textwidth}
		\includegraphics[width=1\textwidth]{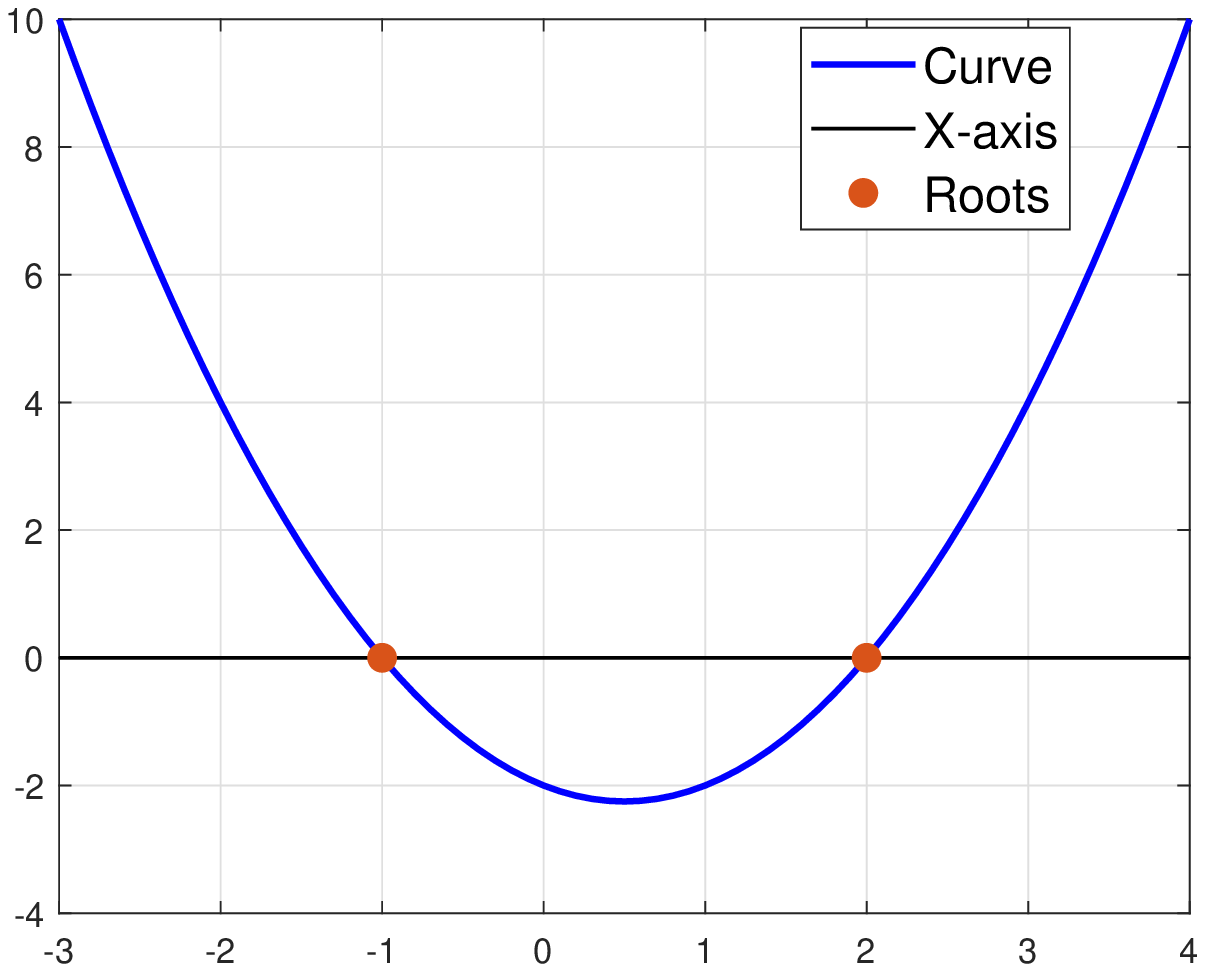}
		\subcaption{}
		\label{odd}
	\end{subfigure}
	\begin{subfigure}[htbp!]{0.39\textwidth}
		\includegraphics[width=1\textwidth]{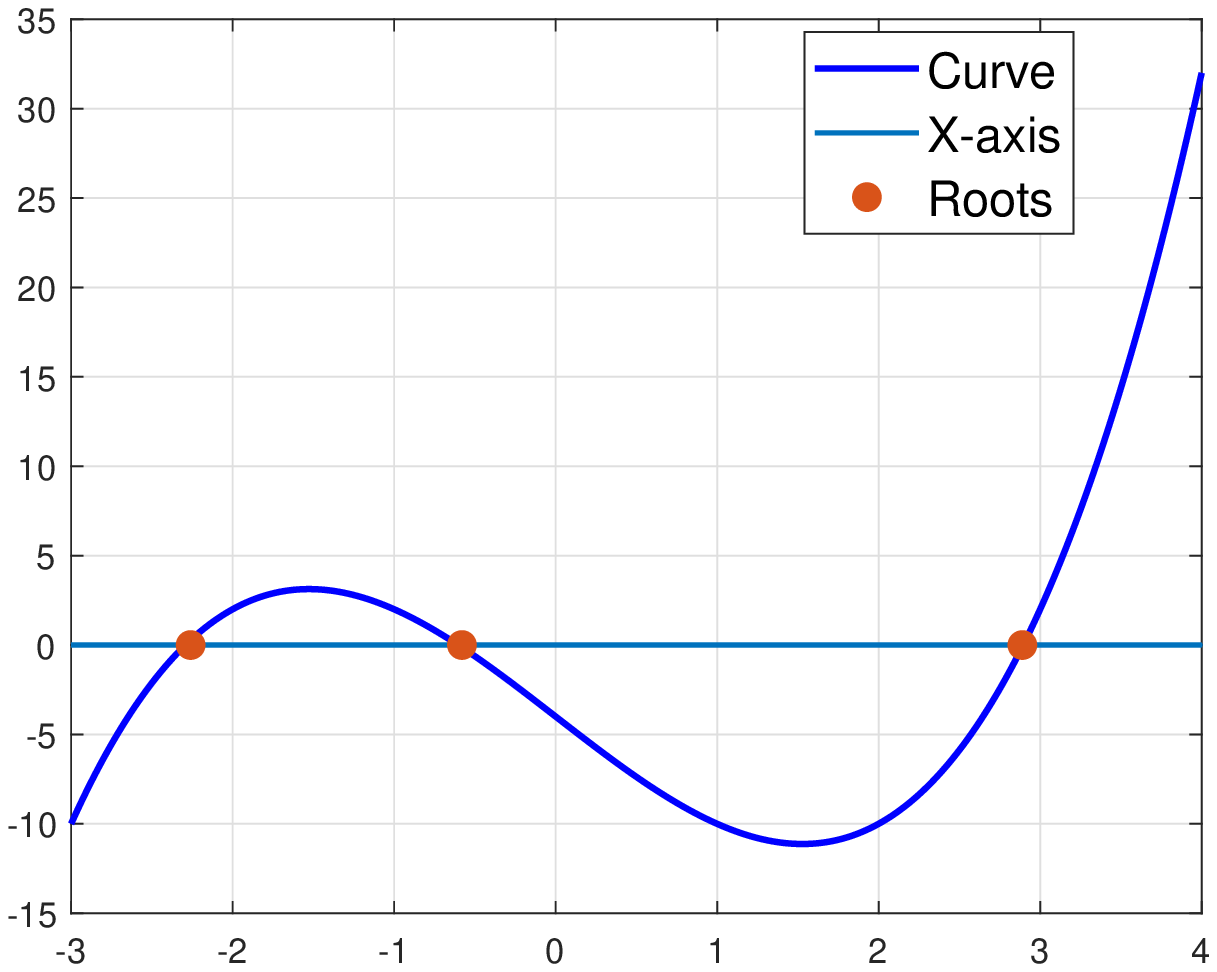}
		\subcaption{}
		\label{even}
	\end{subfigure}
	\caption{
		\quad 
		The  sketch figures   for the curve  of   the  function   $  f(x) = x^{m-1} -\alpha x - \beta$ for odd and even case:  a)   odd  $m$  case: $m=3, \alpha=1, \beta =2$;   b)   even   $m$  case: 
$m=4, \alpha=7, \beta =4$.	}
	\label{curveplot}
\end{figure*}

\begin{remark}
	Note that both 
	$\alpha$ and $\beta$ 
	are the function of $\mathbf u$. 
	For simplicity, we omit the variable  $\mathbf u$ in the notations.
We  discuss the sign of $\alpha $ and  $\beta $  for  different orders $m$. 
When  $m$  is odd, 
since $\alpha(-\mathbf u) = -\alpha(\mathbf u) $,  we can always select  the value where $\alpha >0$. 
For $\beta$ which  is the summation of  $n$ non-negative  terms,  it  holds that  $\beta \ge 0$. 
However, considering the  constraint  in  (\ref{KKTcon1})  and
(\ref{KKTcon2}), all $n$ terms $u_{i}$ cannot be equal to 0 simultaneously.
Therefore, we can  finally  conclude that 
 $\beta > 0$. 
A similar  analysis can be proceeded   for the even $m$  case,  and it  holds that  $\alpha >0$  and  $\beta \ge 0$.
\end{remark}

Then,  as can be  observed  from (\ref{KKTgradient}),  each $u_{i}$   can  be  one of  the roots  of  $(m-1)$-order polynomial 
$ u_{i}^{m-1} -\alpha u_{i} - \beta =0  $.
Note that  in  this  paper,  we  only  discussed  the  real  solutions, consistent with the previous work \cite{teneigenstructure}.
By plotting  the  function  curves 
with  the  form  of 
 $  f(x) = x^{m-1} -\alpha x - \beta$
   for  odd  and  even  two  cases,  
it can be observed  from  Fig 	\ref{curveplot} that 
when $m$ is odd,  there are  at most  two  real  roots  for  $ u_{i}^{m-1} -\alpha u_{i} - \beta =0  $, 
while  $m$ is even,  the number  turns  to be three ones. 
Therefore, throughout  the  following  contents,  we  will  separately discuss  the odd and  even   $m$ cases.
Note that  such  an  observation  and  classification  rule is  consistent with the  discussion  presented in 
the previous work \cite{teneigenstructure}.
The  differences  lie  in  that 
\cite{teneigenstructure}  focuses   on the original model (\ref{optmodelori}), while  
this paper analyzes the  reformulated one (\ref{optmodel}).

And the  following  Lemma  concludes
the structure and number  of all  the  stationary  solutions for model (\ref{optmodel}):
\begin{lemma}\label{Theorem_structureofall}
	
	i):
	when  $m$  is  odd, 
	all  the  stationary   points  of    model (\ref{optmodel}), 
	$	\mathbf   u =  
	[u_1,u_2,\dots, u_{n}]^{\mathrm T} \in  \mathbb {R}^{n \times 1}  $, 
	are  with   the  structure as  follows:
		\begin{equation}\label{u_classifyodd}
\mathbf u
=
\mathbf P
[
a \mathbf 1_{k}^{\mathrm T},
b \mathbf 1_{n-k}^{\mathrm T}
]^{\mathrm T}, 
\end{equation}
	where  
	$ a=   \sqrt{	\frac{n-k} {k n}		} >0 $  and  
	$  b=-  \sqrt{	\frac{k} {(n-k) n}	}   <0 $ vary with $k$, 
	where $1 \le  k  \le \lceil n/2  \rceil $, ($\lceil n/2  \rceil$ denotes the 
	integrate that is no larger than $ n/2 $)
	and
	$\mathbf P \in  \mathbb R^{n \times n} $
	is  a 
	permutation matirx. 
	
			ii):
		when  $m$  is  even, 
		all  the  stationary   points  of    model (\ref{optmodel}), 
		$	\mathbf   u =  
		[u_1,u_2,\dots, u_{n}]^{\mathrm T} \in  \mathbb {R}^{n \times 1}  $, 
		are  with   the  structure as  follows:
		\begin{equation}\label{u_classifyeven1}
		\mathbf u
		=
		\mathbf P
		[
		a \mathbf 1_{k}^{\mathrm T},
		b \mathbf 1_{n-k}^{\mathrm T}
		]^{\mathrm T}
		\end{equation}
		or 
		\begin{equation}\label{u_classifyeven2}
		\mathbf u
		=
			\mathbf P
			[
		c \mathbf 1_{p}^{\mathrm T},
		d \mathbf 1_{q}^{\mathrm T},
		e \mathbf 1_{s}^{\mathrm T}
		]^{\mathrm T}
		\end{equation}
			where  
		$ a=   \sqrt{	\frac{n-k} {k n}		} >0 $  and  
		$  b=-  \sqrt{	\frac{k} {(n-k) n}	}   <0 $ vary with $k$, 
			where $1 \le  k  \le \lceil n/2  \rceil$,  
		and
		$\mathbf P \in  \mathbb R^{n \times n}$
		is  a 
		permutation matirx;
			where $1 \le  p,q,s  \le n-1$ and    $p+q+s=n$.
		Note that there is no   explicit expression for the variables  $c,d,e$
		in (\ref{u_classifyeven2}).
		
			iii): when  $m \ge 3 $ is  odd and $n \ge 3$,  the number of all the stationary points of  model (\ref{optmodel}) is equal to  $	K_{m,n} = 2^{n-1}-1$.
			
		iv): when  $m \ge 3 $ is  even  and $n \ge 3$ ,  the number of all the stationary points of  model (\ref{optmodel}) is equal to 
		\begin{align}
		K_{m,n}=
	\begin{cases}
 \frac 
{ 3^{ n-1} - 1}
{2}  , \quad   n \ge  3, m\ge 4  \\
6  , \quad \quad \quad  n=3, m\ge 6
	\end{cases}.
		\end{align}
		
\end{lemma}

\begin{proof}
	Equivalently,  for (\ref{KKTgradient}), there   are   $n$  equations  with  the  form of 
\begin{equation}\label{uim1alpha} 
u_{i}^{m-1} -\alpha u_{i} - \beta =0  \quad  (i=1,2, \dots, n).
\end{equation}


In term of (i), 
when $m$ is  odd,  there  are  two  real  roots for 
(\ref{uim1alpha}), 
  denoted by $a,b$ for  simplicity. 
It can be  checked that 
all  $n$ variable 
$u_{i}$ cannot  simultaneously 
be  $a$  or  $b$. 
Otherwise, we  will 
 derive    $ u_{1} =u_{2} = \dots = u_{n}$, which  is contradicted   with  the  constraints  in  (\ref{optmodel}).

Therefore, there  are  at  most  $n-1$ 
variable  
$u_{i}$ 
can be the  same  one. 
For  example, 
assume  that  the  first  $k$ ($ 1 \le  k  \le  n-1 $)  
variables  are the  same, 
and 
the last   $n-k$ fractions  are the other one. 
The  form of the  solution  is  given by 
\begin{align}\label{ab_struc}
& u_{1}=u_{2} \cdots=u_{k}  =a  ,  
&u_{k+1}=u_{k+2}  \cdots=u_{n} =b.  
\end{align}



Considering  the  two constraints  in  (\ref{optmodel}), we  have  that 
\begin{equation}\label{ab_expreess}
\begin{cases}
k a^{2}+(n-k) b^{2}=1 \\
k a+(n-k) b=0
\end{cases}. 
\end{equation}
Then, it can be derived that 
\begin{equation}\label{ab_solu}
\begin{cases}
a=   \sqrt{	\frac{n-k} {k n}		} \\
b=-  \sqrt{	\frac{k} {(n-k) n}	}  
\end{cases},
\quad  or  \quad 
\begin{cases}
a=   - \sqrt{	\frac{n-k} {k n}		} \\
b=  \sqrt{	\frac{k} {(n-k) n}	}  
\end{cases},
\end{equation}
where  
$ a, b  $ vary with $k$. 
Furthermore, note  that   because  
 $ f ( -\mathbf  u ) = -  f (\mathbf  u) $  when  $m$ is odd,
we thus can   only reserve   the former  one solution   in (\ref{ab_solu})  where  $ a>0$ .
Then,  the  structure of  all   solutions
can be  denoted  by 
permuting the  indices  of  $i$  by  a  permutation matrix $\mathbf P$, which then  
can be  expressed as  the form in 
(\ref{u_classifyodd}).

In term of (ii), 
when $m$ is  even, 
there  are  at  most  three  real  roots for 
(\ref{uim1alpha}).
Similarly,  all $n$  variables  cannot be the same one, indicating  that there  are  at  least  two  different  values  concerning  the components of $\mathbf  u$.
Then, the  analysis can be similarly  proceeded  as  that  in (i)
and  the structure  of  
the solution will be  given by 
(\ref{u_classifyeven1})
and 
(\ref{u_classifyeven2}).
Note that 
for  (\ref{u_classifyeven1}), 
the value of $a$ and $b$ can be the same as the odd $m$ cases.
While 
for 
(\ref{u_classifyeven2}), 
since we have  three  unknown  variables $c,d,e$
but only  two  constraints, we cannot 
obtain the  explicit  expression  for $c,d,e$.

In term of (iii),   
since  each $u_{i}$ can be  selected  as  one  of  two real  roots,
the total combination numbers   are  given by 
$2^{n}$.
First,  
the  case that  
all $n$   
$u_{i}$ are the same one should be  excluded, as analyzed before. 
In addition, 
in the left 
$2^{n}-2$
combinations, 
it contains  repeated  same solutions, for  example, 
$ [
a \mathbf 1_{k}^{\mathrm T},
b \mathbf 1_{n-k}^{\mathrm T}
]^{\mathrm T} $
and 
$ [
b \mathbf 1_{k}^{\mathrm T},
a \mathbf 1_{n-k}^{\mathrm T}
]^{\mathrm T} $
can be treated as the same one solution. 
This is the reason that we constraint $1 \le  k  \le \lceil n/2  \rceil$,
since the solutions
where $\lceil n/2  \rceil  \le  k  \le  n-1$
can be considered as the same one. 
Therefore, the   total  number of 
feasible solutions 
is  given by 
$
\frac{2^{n}-2}{2}
=
2^{n-1}-1$.

In term of (iv),   for the  even $m$  case,  
it can be analyzed in a  similar way to that in (iii).
Since  one special 
case 	 of 		$n = 3, m = 4$ 
is excluded out of the scope,  as 
explained in Remark \ref{RemarkScope},
the number of all  solutions 
is separately discussed    from two different cases. 
In the first case where 
$ n=3, m\ge 6 $, 
it is easy to conclude that 
there are totally 6  solutions, 
counting three ones for both 
(\ref{u_classifyeven1})
and 
(\ref{u_classifyeven2}).
In the second  case where 
$  n \ge  3, m\ge 4 $, 
similar to the analysis for the odd $m$ case, 
 the  number  of 
all stationary  solutions 
is   
eventually 
given by 
$ \frac 
{ 3^{ n-1} - 1}
{2}  $.
And the proof is complete.
\end{proof}

\begin{remark}
	Note that in 
	(\ref{u_classifyeven1})
	and 
	(\ref{u_classifyeven2}),
	we use five notations  from 
	$a$ to $e$ to distinguish 
	the  two  different   structures   of  the solutions for the even $m$ cases. 
	We wish that  such a  distinction 
	will not cause ambiguity 
	concerning the fact that there are at most three different  real roots for 
		$  u_{i}^{m-1} -\alpha u_{i} - \beta =0  \quad  (i=1,2, \dots, n) $
		when $m$ is even. 
\end{remark}

After  having  investigated the structure of all  stationary  solutions, 
the  following lemma can be  established to 
reveal  the  relationship 
between the vector in the frame (which is the generators of the corresponding regular simplex tensor) and 
all stationary solutions: 
\begin{lemma}\label{vectorregularsimplexframe}
	The  vectors   in the regular simplex  frame, i.e., 
	 $ \mathbf{w}_{1}, \ldots, \mathbf{w}_{n}$,
	 correspond  to 
	 eigenpairs when $k=1$  in both   
	 (\ref{u_classifyodd})  and (\ref{u_classifyeven1}). 
	\end{lemma}

\begin{proof}	
	Concerning  the   vectors   in the regular simplex  frame
 $
 \mathbf W =
 [
 \mathbf{w}_{1}, \ldots, \mathbf{w}_{n}
 ] \in
  \mathbb{R}^{(n-1) \times  n}$, 
  it holds that 
 $ \mathbf{w}_{i}^{\mathrm T} \mathbf{w}_{i} =1, 
  \mathbf{w}_{i}^{\mathrm T} \mathbf{w}_{j} = \alpha = -\frac{1}{n-1} $
  for 
  $ i \neq j$. 
  It can be calculated that  when 
  $ \mathbf v = \mathbf w_{j}$,
  it holds that 
 \begin{align}\label{udenotealpha}
\mathbf   u 
& 
\nonumber =
\sqrt {
	\frac  {n-1}{ n}
}   
{\mathbf W}^{\mathrm T}  {\mathbf w_{j}}  
= 
[u_1,u_2,\dots, u_{n}]^{\mathrm T} 
\\
&=
\sqrt {
	\frac  {n-1}{ n}
} 
[\alpha,  \dots,   \alpha, 
\underbrace{ 1 }_{j}, \alpha,  \dots,  \alpha ]^{\mathrm T} 
\in  \mathbb {R}^{n \times 1} ,
\end{align} 
It can be checked that 
$\mathbf u$ with the above form satisfies the two constraints  in (\ref{optmodel}).
Clearly,  
when $ \mathbf v = \mathbf w_{j}$,
$\mathbf u$  in   (\ref{udenotealpha}) 
correspond to the solutions  of   $k=1$ of 
(\ref{u_classifyodd})  and (\ref{u_classifyeven1}), and 
it holds for both odd  and even cases.
By  the equivalence  between  (\ref{optmodelori}) and  (\ref{optmodel}),
all  the 	  vectors   in the regular simplex  frame, i.e., 
$ \mathbf{w}_{1}, \ldots, \mathbf{w}_{n}$,
correspond  to 
$n$
eigenpairs  of the regular simplex  tensor. 
\end{proof} 

 \section{
 Structure   of  all  locally optimal  points of  (\ref{optmodel})
 }
 \label{locallymaximized}


We have investigated the  structure of all  stationary points of  (\ref{optmodel})  in the above section,
and also shown that 
the vectors in the regular simplex frame 
correspond to 
the solution  of (\ref{optmodel}) when $k=1$  in (\ref{u_classifyodd})  and (\ref{u_classifyeven1}). 
In this section,  we are further interested in identifying  the local optimality of each eigenpair, 
i.e., 
which    are   the  locally  maximized,  minimized or saddle    ones  of (\ref{optmodel})
among    these  solutions.
In other words, we tend to   group 
all these eigenpairs into three classes. 
For this purpose,    it is necessary to   introduce the second-order
derivative information 
of  the  Lagrangian function.      
For  (\ref{Lagrangianfunctionnew}),   the second-order derivation  to    $ \mathbf u $, termed  the Hessian matrix,  is   denoted
\begin{equation}\label{Hessianmatrix}
\triangledown_{\mathbf u\mathbf u}^{2} L
= \mathbf H (\mathbf u)  
=(m-1) diag (\mathbf u ^{\circledast^{m-2}})
-\alpha \mathbf  I_{n}  .
\end{equation}
Similar to 
the  derivation for (\ref{Mhess}),
the 
locally    optimal   solutions of   (\ref{optmodel})
can  be  identified  by
checking the  negative 
definiteness of the  following  matrix:
\begin{equation}\label{Mhess2} 
\mathbf {M} =
(\mathbf  P_{\mathbf  A } ^{\bot})^{\mathrm T}    \mathbf H (\mathbf u)  \mathbf  P_{\mathbf  A }^{\bot}
=
\mathbf  P_{\mathbf  A } ^{\bot}   \mathbf H (\mathbf u)  \mathbf  P_{\mathbf  A }^{\bot}
,
\end{equation} 
where 
$ \mathbf A $
is determined by
the 
two constraints 
$  g_{1}(\mathbf u) $ and  $ g_{2}(\mathbf u) $,
which 
is with the  following form:
\begin{equation}
\label{Aform}
\mathbf A =  [\triangledown  g_{1}(\mathbf u) , \triangledown  g_{2}(\mathbf u) ] =
[\mathbf u, \mathbf 1_{n}]  \in   \mathbb R^{n  \times 2 } .
\end{equation}
And the  detailed  explanation for  the  
derivation of 
(\ref{Mhess2})
are also  presented in  the Appendix part. 



Similarly,
we will separately discuss the cases of odd and even order.
Meanwhile, 
as there are three structure as presented 
from  
	(\ref{u_classifyodd})
to
(\ref{u_classifyeven2}),
these   three different cases are 
separately analyzed in the following subsections.
The details are as follows.

 \subsection{Odd  $m$ case}
 
 In this subsection, we first 
 focus on the tensor with  odd  $m$,
 whose solutions 
 are as shown in 
 (\ref{u_classifyodd}). 
\begin{lemma}\label{Theorem_structureoflocal}
	For the case where   $m$ is odd,  concerning  the solutions 
with the form of 
	(\ref{u_classifyodd}):
	i):
	 when  $k=1$ 
	 the solutions  are   the  locally maximized    points  of    model (\ref{optmodel});
	 ii)  when  $2 \le  k \le \lceil n/2  \rceil $, all the other solutions are  the saddle points  model (\ref{optmodel}). 
	
%
	
\end{lemma}

\begin{proof}
	Due to the form of $\mathbf  A$ as shown in (\ref{Aform}), 
$ \mathbf  P_{\mathbf  A }^{\bot} $ can be   further  rewritten as   
\begin{align}\label{projecteddenote}
\mathbf  P_{\mathbf  A } ^{\bot}
\nonumber
&=
\mathbf  I_{n}  -
\begin{bmatrix}
\mathbf u  &
\mathbf 1_{n}
\end{bmatrix}
(
\begin{bmatrix}
\mathbf u^{\mathrm T} \\
\mathbf 1_{n}^{\mathrm T}
\end{bmatrix}
\begin{bmatrix}
\mathbf u  &
\mathbf 1_{n}
\end{bmatrix}
)
^{-1}
\begin{bmatrix}
\mathbf u^{\mathrm T} \\
\mathbf 1_{n}^{\mathrm T}
\end{bmatrix}
\\ \nonumber
&=
\mathbf  I_{n}  -
\begin{bmatrix}
\mathbf u  &
\mathbf 1_{n}
\end{bmatrix}
(
\begin{bmatrix}
\mathbf u^{\mathrm T}\mathbf u &  \mathbf u^{\mathrm T}\mathbf 1_{n}   \\
\mathbf 1_{n} ^{\mathrm T} \mathbf u& \mathbf 1_{n}^{\mathrm T}\mathbf 1_{n}  \\
\end{bmatrix}
)
^{-1}
\begin{bmatrix}
\mathbf u^{\mathrm T} \\
\mathbf 1_{n}^{\mathrm T}
\end{bmatrix}
\\ 
&=
\mathbf  I_{n}  - 
(\mathbf u\mathbf u^{\mathrm T}  
+ \frac  1n \mathbf 1_{n}\mathbf 1_{n}^{\mathrm T} ). 
\end{align}
where we use the equation   $ \mathbf u^{\mathrm T}\mathbf 1_{n} = \mathbf 1_{n} ^{\mathrm T} \mathbf u =  0 $. 

In addition,   
by  considering  the  structure of   $\mathbf u$ as  shown in  (\ref{u_classifyodd})  and  (\ref{u_classifyeven1}) in Lemma \ref{Theorem_structureofall}, 
$ \mathbf u\mathbf u^{\mathrm T} $  can  be  presented in  the  block  form,  which  follows: 
\begin{align}\label{publock}
\mathbf u\mathbf u^{\mathrm T}
\nonumber
& =
\begin{bmatrix}
a^{2} \mathbf J_{k}   &   ab  \mathbf J_{k \times (n-k)}    \\
ab  \mathbf J_{(n-k) \times k}    &   b^{2} \mathbf J_{(n-k)}
\end{bmatrix}
\\ 
&= 
\begin{bmatrix}
a^{2} \mathbf J_{k} &   -\frac 1n  \mathbf J_{k \times (n-k)}    \\
-\frac 1n  \mathbf J_{(n-k) \times k}    &   	b^{2} \mathbf J_{(n-k)}
\end{bmatrix},	
\end{align}
where based on (\ref{ab_solu}), it holds that   
\begin{equation}\label{abmulti}
ab=   - \frac  1n 
. 
\end{equation} 

Similarly,   
$\frac  1n \mathbf 1_{n}\mathbf 1_{n}^{\mathrm T} $   can also be expressed as 
\begin{align}
\label{p1nblock}
\frac  1n \mathbf 1_{n}\mathbf 1_{n}^{\mathrm T} 
=
\begin{bmatrix}
\frac 1n \mathbf J_{k}  &   \frac 1n  \mathbf J_{k \times (n-k)}    \\
\frac 1n  \mathbf J_{(n-k) \times k}    &   	\frac 1n \mathbf J_{(n-k)}
\end{bmatrix}.   	
\end{align}

Summing (\ref{publock}) and (\ref{p1nblock})  yields
 \begin{align}
 \label{u1nblocksum}
 \mathbf u\mathbf u^{\mathrm T} + \frac  1n \mathbf 1_{n}\mathbf 1_{n}^{\mathrm T}
 \nonumber  
 &=
 \begin{bmatrix}
(a^{2} + \frac 1n) \mathbf J_{k}  &   \mathbf O_{k \times (n-k)}      \\
\mathbf O_{ (n-k) \times k }      &   	(b^{2} + \frac 1n)\mathbf J_{(n-k)}
 \end{bmatrix}
 \\ 
  &=
 \begin{bmatrix}
 \frac 1k \mathbf J_{k}  &   \mathbf O_{k \times (n-k)}      \\
 \mathbf O_{ (n-k) \times k }      &   	\frac{1}{n-k}\mathbf J_{(n-k)}
 \end{bmatrix},
 \end{align}
where  for the second  equation,   we   utilize  
\begin{align}
\notag
   a^{2} + \frac 1n =
(\sqrt{\frac{n-k}{k n}} )^{2} +  \frac 1n
=
\frac{n-k}{k n} + \frac 1n
=
\frac 1k.
\end{align}

 Similarly, we derive that 
\begin{align}
\notag
b^{2} + \frac 1n 
=
\frac{1}{n-k}.
\end{align}

Substitute (\ref{u1nblocksum})  into (\ref{projecteddenote}), and we have that 
\begin{align}\label{pup1n}
 \mathbf  P_{\mathbf  A }^{\bot}  
& =
\begin{bmatrix}
\mathbf I_{k} -\frac 1k \mathbf J_{k} &     \mathbf O_{k \times (n-k)}    \\
\mathbf O_{(n-k) \times k }    &   	\mathbf I_{(n-k)} - \frac {1}{n-k} \mathbf J_{(n-k)}
\end{bmatrix}.	
\end{align}

In a similar way, 
by  considering  the  structure of   $\mathbf u$ as  shown in  (\ref{u_classifyodd})  and  (\ref{u_classifyeven1}) in Lemma \ref{Theorem_structureofall}, 
$\mathbf  H(\mathbf u)$  can   also be  presented in  the  block  form,  which  follows: 
\begin{align}\label{Hblock}
\mathbf  H(\mathbf u)
=
\begin{bmatrix}
 \sigma_{a} 	\mathbf I_{k}     &  \mathbf O_{k \times (n-k)}   \\  
\mathbf O_{(n-k) \times k }   &    \sigma_{b} \mathbf I_{(n-k)}
\end{bmatrix},
\end{align}
where 
based on (\ref{Hessianmatrix}), 
we denote 
$ 
 \sigma_{a} =  (m-1) a^{m-2}-\alpha, 
\sigma_{b}  =(m-1) b^{m-2}-\alpha $
for  simplicity.

Combining  (\ref{Mhess2}), (\ref{pup1n}) and  (\ref{Hblock})  yields
\begin{align}
\mathbf {M}
= 
\begin{bmatrix}
 \sigma_{a} 	(\mathbf I_{k} -\frac 1k \mathbf J_{k})    &  \mathbf O_{k \times (n-k)}   \\  
\mathbf O_{(n-k) \times k }   &      \sigma_{b}	(\mathbf I_{(n-k)} - \frac {1}{n-k} \mathbf J_{(n-k)}) 
\end{bmatrix} . 
\end{align}

It can  be  verified that  $  \mathbf I_{k} -\frac 1k \mathbf J_{k} 
= \mathbf I_{k} -\frac 1k \mathbf 1_{k}  \mathbf 1_{k}^{\mathrm T} $  is  an   idempotent matrix, and  its $k$ eigenvalues are    0  and  1 (the number are $k-1$). 
And  similar  results for  $ \mathbf I_{(n-k)} - \frac {1}{n-k} \mathbf J_{(n-k)}$.
Clearly,  $\mathbf M $  is  the  direct  sum  of  two   matrices, 
and based on the  conclusion in  Lemma  \ref{direct_eig}, 
  $n$  eigenvalues  of  $\mathbf M$   are    
\begin{equation}\label{eigensign} 
\underbrace{
	 \sigma_{a}  ,  \sigma_{a},  \dots,   \sigma_{a} }_{k-1},
\underbrace{
	 \sigma_{b}  ,  \sigma_{b},  \dots,   \sigma_{b}  }_{n-k-1}, 
\underbrace{
0, 0}_{2}. 
\end{equation}

Due to the  fact  that 
each element of  $ \mathbf H _{ii} $ 
is  actually 
the gradient of 
$ u_{i}^{m-1} -\alpha u_{i} - \beta $,
$ \sigma_{a}$ and $ \sigma_{b}$ 
actually  reflect 
the descending or ascending 
trend  at the roots.
When $m$ is odd,  as     can be  observed  from  Fig. \ref{odd}, 
it holds that 
\begin{align}
\mathbf H _{ii} 
&=   
[ (m-1) diag (\mathbf u ^{\circledast^{m-2}})
-\alpha \mathbf  I_{n}  ] _{ii} 
\nonumber 
\\
&=
\begin{cases}
(m-1) a^{m-2}-\alpha = \sigma_{a} >0  ,  \quad  \quad    u_{i}=a  >0    \\
(m-1) b^{m-2}-\alpha = \sigma_{b} <0  ,  \quad  \quad    u_{i}=b    <0  \\
\end{cases}  .
\end{align}

 Then, we can conclude that    
 if and only if  $k=1$,  the  $n$  eigenvalues  of  $\mathbf M$  are all  non-positive, and thus  $\mathbf M$   is  negative  semi-definite,  indicating  that  the corresponding  direction  $\mathbf u $  is  the  locally maximized    one. 
 When  $2 \le  k \le \lceil n/2  \rceil $,  the  $n$  eigenvalues  of  $\mathbf M$  
 contains 
 $k-1$   positive  eigenvalues 
 and  
  $n- k-1$  negative  ones,  indicating  that  the corresponding  direction  $\mathbf u $  is  the  saddle  one. 
\end{proof}

 \subsection{Even   $m$ case with (\ref{u_classifyeven1}) }\label{evencase1}
  In this subsection, we first 
 focus on the tensor with  even   $m$,
 whose solutions 
 are as shown in 
 (\ref{u_classifyeven1}). 
 \begin{lemma}\label{Theorem_structureoflocaleven1}
 	For the case where   $m$ is even,  concerning  the solutions 
 with the form of  (\ref{u_classifyeven1}):
 i):
 	when  $k=1$ 
 	the solutions  are   the  locally maximized    points  of    model (\ref{optmodel});
 	2):  	when  $2 \le  k \le \lceil n/2  \rceil $, denote 
 	\begin{align}
 	l(m,n,k)
 	=
 	(m-1)nk^{m-2} - (n-k)^{m-1}- k^{m-1}
 	\end{align}
 If  $ l(m,n,k) >0$, 
 the  corresponding  solutions  are   the  locally minimized    points  of    model (\ref{optmodel}).
 If 
 $ l(m,n,k) < 0$, 
 the    solutions  are   the  saddle    points  of    model (\ref{optmodel}).

 \end{lemma}

\begin{proof}
Note that even though  for both odd and even $m$ cases, 
the eigenvalue distribution of $\mathbf M$
is with similar structure as shown above  (\ref{eigensign}), 
the sign of    $\sigma_{a}$ and  $\sigma_{b}$
is instead   different for  odd and even $m$ cases.
Different  from  the  conclusion   in the  odd  $m$  case  that  the  sign  of  $\sigma_{a}$ and  $\sigma_{b}$  is  definite,
in  the   even   $m$  case,  we  can  only  determine the  sign  of   $\sigma_{a} > 0$,  while   that of  $\sigma_{b}$  may be  vary  with  $k$,
which     can be  observed from  Fig~\ref{even}.
Due to the  fact  that   $\sigma_{a} > 0$,  it can be  derived  that  when  $k \ge 2$,  there  are  at   least  one  positive   eigenvalue 
$\sigma_{a} > 0$, and  thus  the  corresponding   eigenpair  cannot be  locally  maximized. 
Therefore,  we  only  need  to  analyze the  case  when  $k=1$. 
First,    
$\sigma_{b} $ as  a  function of  $k$  can be  denoted as 
\begin{align}
\sigma_{b}
& =
(m-1)b^{m-2} - \alpha 
\nonumber  \\
&  
=
(m-1)(-  \sqrt{	\frac{k} {(n-k) n}	})^{m-2}  - 
[ k ( \sqrt{	\frac{n-k} { kn}		})^{m} + (n-k) (-  \sqrt{	\frac{k} {(n-k) n}	})^{m} ]
\nonumber  \\
&  =
\frac{ (m-1) k^{r-1}    }
{(n-k)^{r-1}  n^{r-1}   }
-
\frac{  (n-k)^{r}}
{ k^{r-1}  n^{r} }
-
\frac{ k^{r} }
{(n-k)^{r-1}  n^{r}}
\nonumber  \\
&  =
\frac
{  (m-1)nk^{m-2} - (n-k)^{m-1}- k^{m-1}}
{(n-k)^{r-1} k^{r-1}  n^{r}}
\end{align}


The  denominator  will always be positive, i.e., 
$(n-k)^{r-1} k^{r-1}  n^{r} >0 $. 
We only focus on the sign of numerator,  denoted  by
\begin{align}
l(m,n, k)
=
(m-1)nk^{m-2} - (n-k)^{m-1}- k^{m-1}
\end{align}
When $k=1$, the  numerator   equals  to     $
l(m,n, 1) = (m-1)n - (n-1)^{m-1}-1 $. 
Since $n \ge 3$ and $m \ge 4$, and both $  (m-1)n $  and  $ (n-1)^{m-1} $ are  increase function with respect to  $m$ and $n$,  
it holds that 
\begin{align}
l(m,n, 1)
= 
(m-1)n - (n-1)^{m-1}- 1 \le 9-2^{3}- 1 =0
\end{align}
and if and only  if  when $n = 3$ and $m = 4$,  the  equality holds.
Since  we have  excluded the special case of $ (m,n)=(4,3) $
out of the discussed scope (see  Remark  \ref{RemarkScope}  for  details),  
we can conclude that when $k=1$, the eigenvalues  are  definitely  negative, indicating  the  corresponding   eigenpairs   are  locally maximized.
When $2 \le  k \le \lceil n/2  \rceil $,
there are 
always  $k-1$ positive eigenvalues, so the local 
optimality will be determined by the sign of
$l(m,n, k) $,which will vary  for   different combinations for $(m,n)$.
Therefore, in this proof, we do not intend to 
exhaust all cases, but conclude as the content as presented in the above lemma.
\end{proof}

Therefore, 
in both odd  and even  cases, 
we can conclude that  
for the solution 
$ 	\mathbf u
=
\mathbf P
[
a \mathbf 1_{k}^{\mathrm T},
b \mathbf 1_{n-k}^{\mathrm T}
]^{\mathrm T} $   when $k=1$,  
all   the eigenvalue of the  corresponding  matrix  $\mathbf M$ are   negative, and  thus  the  corresponding  eigenpairs are the locally maximized solutions.
This is summarized as the following corollary:
\begin{corollary}\label{coroll}
	All vectors in the regular simplex frame  are the locally maximized solutions of 
	the corresponding model (\ref{optmodelori}).
	\end{corollary}

 \subsection{Even   $m$ case with (\ref{u_classifyeven2}) }
In  this part,  we   analyze  the  sign  of  the  eigenvalues   of  $\mathbf M$    regarding  these  eigenpairs
with  the  following  structure
 as shown in 
(\ref{u_classifyeven2}):
\begin{equation}\label{u3struc}
\mathbf u
=[
c \mathbf 1_{p}^{\mathrm T},
d \mathbf 1_{q}^{\mathrm T},
e \mathbf 1_{s}^{\mathrm T}
]^{\mathrm T}
\end{equation}

Without lossing of generality,  we  assume that  the  following  relationship holds
\begin{equation}
e < d<0 <c. 
\end{equation}

With  this  assumption, the  sign  of  each  diagonal  element  in  the  matrix  $  \mathbf H (\mathbf u)$  can also  be  determined,   which  can be  observed  from  Fig. \ref{even}  and  thus follows:
	\begin{align}\label{eigenclss}
\mathbf H _{ii} 
&=   
[ (m-1) diag (\mathbf u ^{\circledast^{m-2}})
-\alpha \mathbf  I_{n}  ] _{ii} 
\nonumber 
\\
&=
\begin{cases}
(m-1) a^{m-2}-\alpha = \sigma_{c} >0  ,  \quad  \quad    u_{i}=c     \\
(m-1) b^{m-2}-\alpha = \sigma_{d}  <0,  \quad  \quad    u_{i}=d      \\
(m-1) c^{m-2}-\alpha = \sigma_{e}  >0,  \quad  \quad    u_{i}=e      \\
\end{cases}.
\end{align}
Then,  the  eigenvalues are  sorted in  a  ascending  order  as   follows:
\begin{equation}\label{Heigenorder} 
\underbrace{
	\sigma_{d} = \dots, \sigma_{d} }_{q}
<0
\le 
\underbrace{
	\sigma_{c} = \dots, \sigma_{c} }_{p}
\le  
\underbrace{
	\sigma_{e} = \dots, \sigma_{e} }_{s} .
\end{equation}

It  should  be  noted  that   since  we  cannot  obtain  the  explicit  solution  for  $\mathbf u$ as shown in (\ref{u3struc}) in  this  case,  it  will be  difficult  to obtain  all  eigenvalues of $\mathbf M$  as  can be  done    for  the  odd $m$  case  and 
the  case  1  in  Subsection  \ref{evencase1}.
Therefore,  in  this  case,  
we    turn   to  proving    a  weaker   conclusion:  there  is  at  least   one  positive   and  negative  eigenvalue  for $\mathbf M$, which can also 
determine the local  optimality of eigenpairs. 
And the following lemma can be built: 
  
  \begin{lemma}\label{Theorem_structureoflocaleven2}
  	For the case where   $m$ is even,  concerning  the solutions 
  	with the form of  (\ref{u_classifyeven2}),
  except for the case where $q=1$, the other left solutions
  	are   the  saddle    points  of    model (\ref{optmodel}). 
 For $q=1$, it holds that the corresponding solution cannot be a locally maximized one.

  \end{lemma}

\begin{proof}
First,   $  \mathbf {M} $ 
is rewritten
as:
\begin{align}\label{matrix_Mfor3} 
\mathbf {M} 
&=
\mathbf  P_{\mathbf  A } ^{\bot}   \mathbf H (\mathbf u)  \mathbf  P_{\mathbf  A }^{\bot}
=
(\mathbf  I_{n}  - 
\mathbf u\mathbf u^{\mathrm T}  
- \frac  1n \mathbf 1_{n}\mathbf 1_{n}^{\mathrm T} )
 \mathbf H
 (\mathbf  I_{n}  - 
 \mathbf u\mathbf u^{\mathrm T}  
 - \frac  1n \mathbf 1_{n}\mathbf 1_{n}^{\mathrm T} )
\end{align}  



Denote 
\begin{align}\label{Qmat}
\mathbf Q
=
[\mathbf q_{1},  \mathbf q_{2}, \dots,  \mathbf q_{n}] 
\in  
\mathbb R^{n \times n}
\end{align}
is  an  orthogonal  matrix, i.e., $\mathbf Q^{\mathrm T}\mathbf Q = \mathbf Q \mathbf Q^{\mathrm T} =\mathbf I$.
By  further  considering  and  utilizing  the constraint   $  \mathbf u^{\mathrm T}   \mathbf 1_{n} =0$   and  $  \mathbf u^{\mathrm T}   \mathbf u =1$  in the  following,  we  can  then  set 
$\mathbf q_{1} = 
\frac {\mathbf 1_{n}} { \sqrt{n}}$  and  $\mathbf q_{2} =  \mathbf u$.
Since   $ \mathbf Q $ is  an  orthogonal  one,  then  $ \mathbf M $  and   $ \mathbf Q^{\mathrm T}  
\mathbf  M
\mathbf Q $   are   similar  to each  other, which  will be    with  the  same   eigenvalues  according   to  the  conclusion in linear algebra. 


Then,  
the  following  several  relationships  can  be  established.
It  holds  that
\begin{align}\label{uunnQ}
(\mathbf u\mathbf u^{\mathrm T}  
+ 
\frac  1n \mathbf 1_{n}\mathbf 1_{n}^{\mathrm T} )
\mathbf Q
=
[
\frac {\mathbf 1_{n}} { \sqrt{n}} , \mathbf u , \mathbf 0, \dots, \mathbf 0
]
\end{align}

\begin{align}\label{HuunnQ}
\mathbf H
(\mathbf u\mathbf u^{\mathrm T}  
+ 
\frac  1n \mathbf 1_{n}\mathbf 1_{n}^{\mathrm T} )
\mathbf Q
=
[
\frac {\mathbf H \mathbf 1_{n}} { \sqrt{n}} , \mathbf H \mathbf u , \mathbf 0, \dots, \mathbf 0
]
\end{align}

\begin{align}\label{HU} 
\mathbf H
\mathbf u
& =
((m-1) diag (\mathbf u ^{\circledast^{m-2}})
-\alpha \mathbf  I_{n}) \mathbf u
\nonumber  \\
&=
(m-1)  \mathbf u ^{\circledast^{m-1}}-\alpha \mathbf u
\nonumber  \\
& =
(m-1) \alpha  \mathbf u + (m-1) \beta \mathbf  1_{n}-\alpha \mathbf u 
\nonumber  \\
& =
(m-2) \alpha  \mathbf u + (m-1) \beta \mathbf  1_{n}
\end{align}  
where  (\ref{KKTgradient}) is utilized in the third rquation.
Similarly, we have that 
\begin{equation}\label{HLN} 
\mathbf H
\mathbf  1_{n}
=
(m-1)  \mathbf u ^{\circledast^{m-2}}-  \alpha  \mathbf  1_{n}
\end{equation}  

Furthermore, it  can be  derived  that 
\begin{align}\label{UHU} 
\mathbf u^{\mathrm T}
\mathbf H
\mathbf u
=
(m-2) \alpha,  
\quad
\mathbf  1_{n}^{\mathrm T}
\mathbf H
\mathbf  1_{n}
=
(m-1) \gamma
-
n  \alpha  ,
\quad
\mathbf u^{\mathrm T}
\mathbf H
\mathbf 1_{n}
=
\mathbf  1_{n}^{\mathrm T}
\mathbf H
\mathbf  u
=
n(m-1) \beta,
\end{align} 
where
$ \gamma
=
\mathbf 1_{n}^{\mathrm T}  
\mathbf u ^{\circledast^{m-2}}
=
\sum u_{i}^{m-2}
>0
$
for  even  $m$  case, and 
we use 
that 
$ \mathbf u^{\mathrm T}  
\mathbf u ^{\circledast^{m-2}}
=
n \beta
$.

Substitute   (\ref{uunnQ}) $\sim$ (\ref{UHU}) into 
$
\mathbf Q^{\mathrm T}  
\mathbf   M
\mathbf Q
$, and 
we can obtain that 
\begin{align}\label{blockdiag0}
\mathbf Q^{\mathrm T}  
\mathbf   M
\mathbf Q
=
	\left[\begin{array}{cc}
	\mathbf {O_{2}} & \mathbf {O}_{2 \times (n-2)} \\
	\mathbf {O}_{(n-2) \times 2} & 
	\mathbf Z^{\mathrm T}  
	\mathbf H
	\mathbf Z
\end{array}\right]
\end{align}
where 
\begin{align}\label{Hblockfor3}
\mathbf   Z
=
\mathbf   Q_{:,3:n}
=
[\mathbf q_{3},  \mathbf q_{4}, \dots,  \mathbf q_{n}] 
\in  
\mathbb R^{n \times (n-2)}
\end{align}
is a column-orthogonal matrix 
that satisfies 
$\mathbf Z^{\mathrm T}   \mathbf   Z 
=
\mathbf I_{n-2}$.  

Assuming  that 
the $n-2$ eigenvalues of 
$ 	\mathbf Z^{\mathrm T}  
\mathbf H
\mathbf Z $
are  denoted  by 
$
\lambda_{1}, \lambda_{1},  \dots, \lambda_{n-2}$.
Based on  the conclusion  in 
Lemma  \ref{direct_eig},  the $n$  eigenvalues  of  
$\mathbf Q^{\mathrm T}  
\mathbf   M
\mathbf Q$ 
and   
$
\mathbf   M
$ 
can  be  given  by 
$
\lambda_{1}, \lambda_{1},  \dots, \lambda_{n-2}, 0, 0$.
Therefore,  the  following  task  is to  determine  the 
$n-2$ eigenvalues of 
$ 	\mathbf Z^{\mathrm T}  
\mathbf H
\mathbf Z $.
However,  
this is  still a   tough  problem.
By  utilizing  the  conclusion  in  Lemma 
\ref{AB_BA_eig}, 
we   further  analyze  the  eigenvalues   distribution  of 
\begin{align}\label{hzzt}
\mathbf H
\mathbf Z
	\mathbf Z^{\mathrm T} 
	=
\mathbf H
 (\mathbf  I_{n}  - 
\mathbf u\mathbf u^{\mathrm T}  
- \frac  1n \mathbf 1_{n}\mathbf 1_{n}^{\mathrm T} )
=
\mathbf H
  - 
\mathbf H
(\mathbf u\mathbf u^{\mathrm T}  
+\frac  1n \mathbf 1_{n}\mathbf 1_{n}^{\mathrm T} ).
\end{align}

Definitely,  its   $n$  eigenvalues  are  also   given  by 
$
\mathbf \Sigma=
diag( 0, 0, \lambda_{1}, \lambda_{2},  \dots, \lambda_{n-2}) $.

Since  $\mathbf H$  has  been  a  diagonal  matrix,  and its  eigenvalues  are  easy  to be  determined  as  presented  in (\ref{Heigenorder}). 
So,  the  following  main  task  is  to  determine   the  eigenvalues  of  $
 -
\mathbf H
(\mathbf u\mathbf u^{\mathrm T}  
+\frac  1n \mathbf 1_{n}\mathbf 1_{n}^{\mathrm T} )  
$.
A  similar  orthogonal  transformation  by  $\mathbf Q$  is  also  first  performed  on  this   matrix  and 
eventually, 
  the  matrix  $
  -\mathbf Q^{\mathrm T}  
  \mathbf H
  (\mathbf u\mathbf u^{\mathrm T}  
  +\frac  1n \mathbf 1_{n}\mathbf 1_{n}^{\mathrm T} )
  \mathbf Q  
  $  is  with the  following   structure:
\begin{align}\label{QHuuQ}
  -\mathbf Q^{\mathrm T}  
\mathbf H
(\mathbf u\mathbf u^{\mathrm T}  
+\frac  1n \mathbf 1_{n}\mathbf 1_{n}^{\mathrm T} )
\mathbf Q  
=
\begin{bmatrix}
g_{11}    &  g_{12}  &  0  &  \dots  &   0  \\  
g_{21}    &  g_{22}  &  0  &  \dots  &   0  \\  
g_{31}  &  0    &  0  &  \dots  &   0      \\
\vdots                 \\
g_{n1}  &  0    &  0  &  \dots  &   0 
\end{bmatrix}
:=
\mathbf G
\end{align}
where
$ g_{11} = 
 \alpha - \frac 
 {(m-1)\gamma}
 {n} , g_{22} =-(m-2)\alpha <0, 
g_{12} = g_{21} = -\sqrt {n}(m-1)\beta  \le 0 $  for  even $m$  cases, 
$g_{1j} = g_{j1} = -\frac{  (m-1) \mathbf q_{j}^{\mathrm T} \mathbf u ^{\circledast^{m-2}} }{ \sqrt{n}} , j=3,4,\dots, n$,


With  this  form,  we then   analyze  the  sign  of  the eigenvalue of $\mathbf G$, denoted  in  a  descend  order  $\varepsilon_{1} \le  \varepsilon_{2}\le \dots  \le \varepsilon_{n}$,  which  can  be  uniformly  transformed  into  finding   the  roots of  the  characteristic polynomial of 
\begin{align}\label{charpoly}
f(\varepsilon)
&
=
det(\mathbf {G} -  \varepsilon \mathbf {I} )
 \nonumber  \\
&=
(g_{22} - \varepsilon)
(- \varepsilon)^{n-2}
det(
g_{11} - \varepsilon
-
\begin{bmatrix}
g_{21}    \\  
g_{31}      \\
\vdots                 \\
g_{n1}  
\end{bmatrix}
^{\mathrm T}
\begin{bmatrix}
 g_{22} - \varepsilon  &  0  &  \dots  &   0  \\  
0  &  - \varepsilon    &  0  &  \dots  &   0      \\
\vdots                 \\
g_{n1}  &  0    &  0  &  \dots  &  - \varepsilon 
\end{bmatrix}^{-1}
\begin{bmatrix}
g_{12}    \\  
0     \\
\vdots                 \\
0  
\end{bmatrix}
)
 \nonumber  \\
&=
(g_{22} - \varepsilon)
(- \varepsilon)^{n-2}
(
g_{11} - \varepsilon
-
\frac{g_{12}^{2}}{g_{22} - \varepsilon}
)
 \nonumber  \\
&=
-
(- \varepsilon)^{n-2}
(
 \varepsilon^{2}
-
(g_{11} +g_{22}) \varepsilon
+
(g_{11} g_{22} 
-
g_{12}^{2}
)
\end{align}
where
  in  the second equation,  we 
block the matrix  into $1 \times 1 $ and  $(n-1) \times (n-1)  $  and  then utilize  (\ref{detblock}).

When 
$ f(\varepsilon) =0$, it  can be  observed
from (\ref{charpoly})  that  it  has  $(n-2)$-fold  roots  at $\varepsilon=0$.
Therefore,  in  the  following, we further  need  to  determine  the sign  of  the  left  two  non-zero   roots, which can  be  calculated  by  
\begin{align}
z(\varepsilon) 
=
\varepsilon^{2}
-
(g_{11} +g_{22}) \varepsilon
+
(g_{11} g_{22} 
-
g_{12}^{2}
)
\end{align}


Clearly,  $ z(\varepsilon) $ is  an
univariate quadratic equation.
On the  one  hand, 
 the   axis of symmetry
is  
given  by  
$g_{11} +g_{22}$.
Since 
\begin{align}\label{gammaalpha}
\gamma
&=
  \sum\limits_{i=1}^{n}  u_{i}^{m-2}
  =
 (  \sum\limits_{i=1}^{n}  u_{i}^{m-2}) 
 (  \sum\limits_{i=1}^{n}  u_{i}^{2}) 
 =
 \sum\limits_{i=1}^{n}  u_{i}^{m}
 +
 \sum\limits_{i,j=1}^{n}  u_{i}^{m-2} u_{j}^{2}
 =
 \alpha + \theta
\end{align} 
where 
$  \theta >0$  for  even  $m$  cases due  to  each  summation  term in   $ \theta $ is a  positive one.
It can be checked that 
\begin{align}\label{trg}
 tr(\mathbf G)
 &= g_{11} +g_{22}
=
  \alpha - (m-1)\gamma/n -  (m-2)\alpha
 \nonumber \\
 &=
- ( \frac {m-1}{n} +m-3) \alpha - 
 \frac {m-1}{n}  \theta
<0, 
 \end{align}

On the  other  hand,     the discriminant
can be  computed  by 
   $ 
   \Delta = (g_{11} +g_{22})^{2} - 4(g_{11} g_{22} - g_{12}^{2} ) 
   =
   (g_{11} -  g_{22})^{2} +4 g_{12}^{2}   \ge 0$. 
Therefore, 
the  other   two  non-zero  eigenvalues  of  
$ \mathbf G$  
are  given  by 
 \begin{align}
\varepsilon_{1} = 
 \frac{ (g_{11}+g_{22}) - \sqrt{\Delta}}{2}, 
 \varepsilon_{2}
 =
  \frac{ (g_{11}+g_{22})  + \sqrt{\Delta}}{2}.
  \end{align} 
In  addition,  
based on   the  fact  that  the  axis of  symmetric  is negative  by (\ref{trg}), we also  have  that  
$ \varepsilon_{1} < 0$, while the sign of 
$ \varepsilon_{2} $ 
cannot be  determined.
Furthermore, 
since 
\begin{align}
z(g_{11}) 
=
z(g_{22}) 
=
-
g_{12}^{2}
<0 ,
\end{align}
it also  holds  that 
$
\varepsilon_{2}
> 
g_{11}$, 
and 
$
\varepsilon_{2}
> 
g_{22}$.

So  far,  we  have  
analyzed the  eigen-structure
of  $\mathbf G$,  and 
according to 
(\ref{QHuuQ}),
the  $n$  eigenvalues  of 
$
-
\mathbf H
(\mathbf u\mathbf u^{\mathrm T}  
+\frac  1n \mathbf 1_{n}\mathbf 1_{n}^{\mathrm T} )  
$
can be  naturally  determined,  which  contains 
$n-2$ zero  roots  and 
2  non-zero  ones  
$ \varepsilon_{1}, 
\varepsilon_{2} $.
We  sort    these   roots  in  a  ascending  order   as       follows,  
which may  have  two  cases:
\begin{equation}\label{case1} 
Case \quad  1:  \varepsilon_{1}
<
\varepsilon_{2}
<
\underbrace{
	0, 0, \dots, 0}_{n-2}
\end{equation}
or 
\begin{equation}\label{case2} 
Case \quad  2:  \varepsilon_{1}
<
\underbrace{
	0, 0, \dots, 0}_{n-2}
<
\varepsilon_{2}
\end{equation}



Then, by  (\ref{hzzt}),  we
trun  to  analyzing 
the  eigenvalues  
and 
tend to 
show that 
there are at least one positive  and negative  eihenvalues for 
$
\mathbf H
\mathbf Z
\mathbf Z^{\mathrm T} 
$, 
when  we  have  clearly known  the 
eigenvalues  distribution of
its  divided two parts
$
\mathbf H
$ (as listed in (\ref{Heigenorder}))
and
$
-
\mathbf H
(\mathbf u\mathbf u^{\mathrm T}  
+\frac  1n \mathbf 1_{n}\mathbf 1_{n}^{\mathrm T} )  
$ (as listed in (\ref{case1})  or (\ref{case2})).
Two cases are discussed as follows:


\textbf{Case 1:}
Since $p \ge 1,q \ge 1,s \ge 1$ 
and $p +q +s =n$,  
the  number  of    positive  eigenvalues   for  matrix $
\mathbf H
$
holds  that 
$ p+s \ge 2$.
We  first   focus on  the case  1  shown in 
 (\ref{case1}).
  by  using the Weyl  theorem  in  Lemma  \ref{weyltheo},  
  there is also  at  least  one  negative  eigenvalue.
 When 
 $ p+s \ge 3$, 
 by  using the Weyl  theorem  in  Lemma  \ref{weyltheo},  
 there is also  at  least  one  positive  eigenvalue
 for 
 $
 \mathbf H
 \mathbf Z
 \mathbf Z^{\mathrm T} 
 $.
 The  only  left   one  situation  that  need to be  proved  is  
 the  case  of 
  $ p+s =2$ (equivalently, $ p=1,s=1, q=n-2$). 
 In this case, there are only  two positive  eigenvalues
  $  \sigma_{c}  \le  \sigma_{e} $ for 
 the matrix $\mathbf H$, and the two non-zero  roots 
 $\varepsilon_{1} \le  \varepsilon_{2}$ for 
 $
 -
 \mathbf H
 (\mathbf u\mathbf u^{\mathrm T}  
 +\frac  1n \mathbf 1_{n}\mathbf 1_{n}^{\mathrm T} )  
 $.

The  following  inequality  holds that  
  \begin{align}
 n  \sigma_{d}    <  tr(\mathbf H)=
 (m-1) \gamma
 -
 n  \alpha  
 =
 {p} \sigma_{c} + 
 {q} \sigma_{d} + 
 {s} \sigma_{e} 
 < 
 n  \sigma_{e} 
 \end{align} 
we can  conclude 
 \begin{align}\label{sigmac}
\sigma_{c}  + 
\varepsilon_{2}
&
>
\frac{tr(\mathbf H)}{n}
+g_{11}
\ge 0
\end{align} 


For  the  case  $m=4$,  we can further  have that 
\begin{equation}\label{cde}
(u_{i}-c)
(u_{i}-d)
(u_{i}-e)
=
u_{i}^{m-1} -\alpha u_{i} - \beta 
\end{equation}
which will  deduce that 
$ c+d+e=0$.
Then,  
solving the  following system  
\begin{equation}\label{ab_expreessfor3}
\begin{cases}
pc^{2}+q d^{2} +s e^{2} =1 \\
pc+q d +s e =1\\
c+d+e=0
\end{cases}. 
\end{equation}
will  derive  that 
$  
d=0$,
 and 
 $ 
 c= -e,  \sigma_{c} =  \sigma_{e}
 $. 
Combining  
with (\ref{sigmac}),  it holds that 
$\sigma_{e}  + 
\varepsilon_{2} \ge   0$.
 Using  the Weyl  theorem  in  Lemma  \ref{weyltheo},
 it can be concluded that 
 there is also  at  least  one  positive  eigenvalue
 for 
 $
 \mathbf H
 \mathbf Z
 \mathbf Z^{\mathrm T} 
 $.

 For the even $m\ge 6$ cases, we cannot directly   obtain 
 the  new  constraint 
 $ c+d+e=0$ by (\ref{cde}).
 However, 
 as can be seen from 
 Fig 	\ref{curveplot}
 and 
 Lemma  \ref{Theorem_structureofall},
 the numer of all eigenpairs 
 is only determined by 
 dimension $n$
 in both odd and  even $m$  cases, 
 and 
 this indicates that 
for the even $m\ge 6$ cases, 
the same conclusion with the $m=4$ 
case can be arrived, which will also
 conclude that 
 $\sigma_{e}  + 
 \varepsilon_{2} \ge   0$,
 indicating that 
  there is also  at  least  one  positive  eigenvalue
 for 
 $
 \mathbf H
 \mathbf Z
 \mathbf Z^{\mathrm T} 
 $.
 So far, we have shown that 
 there are at  least one positive and negative eigenvalue
 for 
  $
 \mathbf H
 \mathbf Z
 \mathbf Z^{\mathrm T} 
 $ 
 in all  conbinations 
 for case 1.

 \textbf{Case 2:}
The analysis of case  2  is very similar to that for case 1.  
 Clearly, for  the  case 2  shown  in  (\ref{case2}), 
 since the  number  of    positive  eigenvalues   for  matrix $
 \mathbf H
 $
 holds  that 
 $ p+s \ge 2$,
 by  using the Weyl  theorem  in  Lemma  \ref{weyltheo},  
 there is always  at  least  one  positive  eigenvalue
 for 
 $
 \mathbf H
 \mathbf Z
 \mathbf Z^{\mathrm T} 
 $.
When 
$q > 2$, 
by  using the Weyl  theorem  in  Lemma  \ref{weyltheo},  
there is also  at  least  one  negative  eigenvalue
for 
$
\mathbf H
\mathbf Z
\mathbf Z^{\mathrm T} 
$.
The  only  left   one  situation  that  need to be  proved  is  
the  case  of 
$ q=1$. 
For $ q=1$,  it still holds that 
 there is always  at  least  one  positive  eigenvalue, indicating 
 that it cannot be a locally maximized one. 
 To determine whether it is a  saddle or locally minimized one remains unsolved. 
  The proof  for the lemma  is  complete.
 
\end{proof}

\section{Future Work}\label{futurework}
The regular simplex tensor
is a  newly-emerging 
concept and has received  some  attentions since its proposal.   
In  this part, we  will  discuss  some  interesting issues that are worth studying in the future.

1. 
One widely researched issue concerning the tensor eigenpair concept
is 
to identify whether the  eigenpairs  obtained by the tensor power method  is  robust or not. 
The detailed definition for robust eigenpairs can refer to Theorem 3.2 of Ref  \cite{RobustEigen}.
Concerning this subject, it has been discussed in  Ref  \cite{RobustEigen}
but is not comprehensively 
addressed,
which was claimed as the following conjecture:
	\begin{conjecture}[Conjecture 4.7 in  Ref  \cite{RobustEigen} ] \label{conjecturesimplex}
	The robust eigenvectors of a regular simplex tensor   are precisely the
	vectors in the frame.
\end{conjecture}
A rigorous and  complete   theoretical proof  is  one of the  following  tasks to be paid attentions. 


2. Consider  the  following  more  generalized   construction of   regular simplex tensor
	\begin{equation}
\mathcal{S}:=\sum_{i=1}^{n} 
\lambda_{k} \mathbf{w}_{i}^{\circ m}
\end{equation}
where the  weight  factors $\lambda_{k}$  is   included  in  each  term,  which we 
would like to term as the weighted regular simplex tensor.
In this  paper,  
we only consider  one  special case where all  $\lambda_{k} $  are  equal to 1, as defined in 
(\ref{simplextensor}).
The  above weighted version could be  more generalized.
What is the eigen-structure of  its  eigenpairs?
Whether the   conclusion  in conjecture  \ref{conjecturesimplex}  
can be  extended into the generalized  weighted regular simplex tensor?
These  remain    open problems.

3. The other conjectures  claimed in the original reference \cite{RobustEigen}
are also worth investigating  in the next  stage. 
One can refer to Conjecture 4.8 and 5.2 for details.

\section{Appendix}

In this  appendix part,   
we will provide some important theorems   and lemmas 
 used in our proof.
 In addition,  
the detailed  explanations for the second-order necessary  condition is also presented.

\begin{lemma}(Ref \cite{zhang2017matrix})\label{direct_eig}
	Suppose  that 
	$ (\lambda_{i},  \mathbf u_{i})$  
	($i=1,2,\dots,n$),
	$  (\sigma_{j},  \mathbf v_{j})$  
	($j=1,2,\dots,m$)  
	are  eigenpairs of  
	$ \mathbf A$ and $ \mathbf B$, respectively. 
	The ($n+m$)  eigenpairs   of  
	$ \mathbf  {A} \oplus \mathbf  {B}$  
	is  
	$ (\lambda_{i},  
	\begin{bmatrix}
	\mathbf u_{i}   \\
	\mathbf 0_{m}    
	\end{bmatrix}	)$, 
	$ (\sigma_{j},  
	\begin{bmatrix}
	\mathbf 0_{n}   \\
	\mathbf v_{j}
	\end{bmatrix}	)$ 	($i=1,2,\dots,n$, $j=1,2,\dots,m$).
\end{lemma}

\begin{lemma}(Ref \cite{zhang2017matrix})\label{detblocklemma}
	Assume that   a  matrix  with a  size  of  $(n+m) \times(n+m)$,   
	can be  further  divided  into      $ n \times n$  and     $m \times m$   two blocks,
	its  determinant  can be calculated by 
	\begin{equation}\label{detblock}
	det \left[\begin{array}{ll}
	\mathbf {A} & \mathbf {B} \\
	\mathbf {C} & \mathbf {D}
	\end{array}\right]
	=det  (\mathbf {D})
	det\left(\mathbf {A}-\mathbf {B} \mathbf {D}^{-1} \mathbf {C}\right)
	\end{equation}
	when  the  matrix  $\mathbf D$ is 
	assumed to be an  inverse one.
\end{lemma}

\begin{lemma}[\textbf{Weyl Theorem}]\label{weyltheo}
Assume  that 
$\mathbf A$, $\mathbf B \in \mathbb R^{n \times n}$ are Hermitian  matrix,  and  the  eigenvalues are  sorted  in  a  ascending  order. 
\begin{equation}
\begin{aligned}
\lambda_{1}(\mathbf{A}) & \le \lambda_{2}(\mathbf{A}) \le \cdots \le \lambda_{n}(\mathbf{A}) \\
\lambda_{1}(\mathbf{B}) & \le \lambda_{2}(\mathbf{B}) \le \cdots \le \lambda_{n}(\mathbf{B}) \\
\lambda_{1}(\mathbf{A}+\mathbf{B}) & \le \lambda_{2}(\mathbf{A}+\mathbf{B}) \le \cdots \le \lambda_{n}(\mathbf{A}+\mathbf{B})
\end{aligned}
\end{equation}

Then,  the  following  inequalities   hold:
\begin{equation}
\lambda_{i}(\mathbf{A}+\mathbf{B}) \geqslant\left\{\begin{array}{c}
\lambda_{i}(\mathbf{A})+\lambda_{1}(\mathbf{B}) \\
\lambda_{i-1}(\mathbf{A})+\lambda_{2}(\mathbf{B}) \\
\vdots \\
\lambda_{1}(\mathbf{A})+\lambda_{i}(\mathbf{B})
\end{array}\right.
\end{equation}
and 
\begin{equation}
\lambda_{i}(\mathbf A+\mathbf B) \le\left\{\begin{array}{c}
\lambda_{i}(\mathbf{A})+\lambda_{n}(\mathbf B) \\
\lambda_{i+1}(\mathbf{A})+\lambda_{n-1}(\mathbf{B}) \\
\vdots \\
\lambda_{n}(\mathbf{A})+\lambda_{i}(\mathbf{B})
\end{array}\right.
\end{equation}
for  $i=1,2,\dots, n$.
\end{lemma}

\begin{lemma}(Ref \cite{zhang2017matrix})\label{AB_BA_eig}
Given  
the    $ m \times n$    matrix  
$\mathbf A$
and   the   $n \times m$   matrix  
$\mathbf B$ ( $m < n $),
assuming  that 
the $m$ eigenvalues of 
$ 	\mathbf A
\mathbf B $
are  denoted  by 
$
\lambda_{1}, \lambda_{2},  \dots, \lambda_{m}$,
then  the $n$  eigenvalues  of  
$ 	\mathbf  B
\mathbf A $
are  given  by 
$
\lambda_{1}, \lambda_{2},  \dots, \lambda_{m}, 
0, 0$.
\end{lemma}

The  following  theorem is  well  established     for   the   constrained  optimization  problem   to  identify  the  locally  optimal   solutions
(Page 332 in    \cite{Numerical}): 

\begin{theorem}[\textbf{Second-order necessary condition}]\label{second_order_necessary}\cite{Numerical}
	Suppose that
	for any  vector $ \mathbf w \in \mathbb V $,
	if 
	\begin{equation}\label{second_order}
	\mathbf w^{\mathrm T}
	\mathbf H (\mathbf v) 
	\mathbf w  
	\le 0   
	\end{equation}
	holds, then
	$\mathbf v $
	is a local maximum solution of (\ref{opti_ori}).
	And for  (\ref{opti_ori}), the set
	$\mathbb V $
	is defined as
	\begin{equation}\label{vdefine}
	\mathbb V=\{
	\mathbf w \in \mathbb R^{n}
	\vert   (\triangledown  g)^{\mathrm T} \mathbf w   =0
	\}=
	\rm Null [(\triangledown  g)^{\mathrm T} ]
	=
		\rm Null [\mathbf A ],
	\end{equation}
	where   $ \rm Null( \mathbf A) $ 
	denotes the null space  of $\mathbf A$  and $  \triangledown  g =\mathbf v $  denotes the gradient of the constraint: $ g (\mathbf v) = 
	\mathbf v^{\mathrm T}
	\mathbf v 
	-1
	=0 $.  
	
	If a stronger condition, i.e.,
	$ \mathbf w^{\mathrm T}
	\mathbf H (\mathbf v) 
	\mathbf w  <   0 $,   
	is satisfied, then,  
	$\mathbf v$
	is a   strict     local maximum  solution of (\ref{opti_ori}).
\end{theorem}

In  practice, instead of  directly utilizing  Theorem 
\ref{second_order_necessary},  
it  is  more  preferred to  identify the  local  extremum  by  checking  the  positive or negative  definiteness of  the  projected  Hessian  matrix (denoted 
$\mathbf P 
$).  
It  is  calculated  by  
\begin{equation}\label{Pmatrix}
\mathbf P = 
\mathbf Q_{2}^{\mathrm T}
\mathbf H(\mathbf  v) 
\mathbf Q_{2},  
\end{equation}
where 
$  \mathbf Q_{2} $ is  obtained by
QR  factorization of   $     \triangledown  g $:
\begin{equation}\label{QR_factor}
\begin{split}
\triangledown  g
&=\mathbf Q
\begin{bmatrix}
\mathbf R   \\
\mathbf 0
\end{bmatrix} =
\begin{bmatrix}
\mathbf Q_{1}  &   \mathbf Q_{2}
\end{bmatrix}
\begin{bmatrix}
\mathbf R  \\
\mathbf 0
\end{bmatrix}
=\mathbf Q_{1}
\mathbf R
\end{split},
\end{equation}
where 
$  \mathbf Q $ 
is   an     
orthogonal  matrix, and $ \mathbf R$ is a square upper triangular matrix. 
In this  case, 
$ \mathbf R$ is  reduced  to  a  scalar  since  
$ \triangledown  g $  is  a  column  vector  and  
$\mathbf 0$ is  an    vector  with  all  elements  equal to 0.  
See more  details for  this  part in  Page 337 of  \cite{Numerical}. 

Furthermore,  based on  (\ref{vdefine}), we  can  conclude  that  $   \mathbf w$  must  also  lie  in  the  orthogonal complement space of $ \mathbf  v$. 
Therefore,  $   \mathbf w$ can be linearly expressed  by the column vectors of  $\mathbf  P_{\mathbf  v } ^{\bot}$.
Assume that  $ \mathbf z \in \mathbb R^{(n-1) \times 1} $ is the coefficient vectors, it holds that  
$   \mathbf w   =  \mathbf  P_{\mathbf  v } ^{\bot}  \mathbf  z$.  
Then it can be  verified   that 
$  \mathbf v^{\mathrm T} \mathbf w 
=   \mathbf v^{\mathrm T}   \mathbf  P_{\mathbf  v } ^{\bot}  \mathbf  z
=  
\mathbf v^{\mathrm T} (\mathbf  I_{n} -
\mathbf  v (\mathbf  v^{\mathrm T}\mathbf  v)^{-1} \mathbf  v^{\mathrm T} )
\mathbf  z
= \mathbf  0
$  holds.  
Therefore,   (\ref{second_order}) can be  rewritten as 
$ 	\mathbf w^{\mathrm T}\mathbf H (\mathbf v) \mathbf w  
=
\mathbf  z^{\mathrm T}    (\mathbf  P_{\mathbf  v } ^{\bot})^{\mathrm T}    \mathbf H (\mathbf v)  \mathbf  P_{\mathbf  v }^{\bot}  \mathbf  z     \le 0 $.
Since  $ \mathbf  z $ can be   arbitrary  vectors,  checking  (\ref{second_order}) 
is equivalent to  checking  negative semi-definiteness    of  the  matrix, denoted  
$\mathbf  M $, with the  form of 
\begin{equation}
\notag
\mathbf {M} =
(\mathbf  P_{\mathbf  v } ^{\bot})^{\mathrm T}    \mathbf H (\mathbf v)  \mathbf  P_{\mathbf  v }^{\bot}
=
(\mathbf  P_{\mathbf  v } ^{\bot})    \mathbf H (\mathbf v)  \mathbf  P_{\mathbf  v }^{\bot}
,
\end{equation}  
which  
corresponds to 
	(\ref{Mhess}).

Note that 
there are only one constraint for model (\ref{opti_ori}).
When more constraints are included, it can be similarly analyzed. 
For example. 
for    model  (\ref{optmodel}) with two constraints,  $\mathbf A$  should be  defined  as  
\begin{equation}
\notag
\mathbf A =  [\triangledown  g_{1}(\mathbf u) , \triangledown  g_{2}(\mathbf u) ] =
[\mathbf u, \mathbf 1_{n}]  \in   \mathbb R^{n  \times 2 } .
\end{equation}
And 
the 
locally    optimal   solutions of   (\ref{optmodel})
can  be  identified  by
checking the  negative 
definiteness of the   matrix  as  defined  in (\ref{Mhess2}).

\bibliographystyle{siamplain}
\bibliography{simplexref}

\end{document}